\documentclass[10pt,oneside,a4paper]{article} 	

\usepackage[fleqn]{amsmath} 						
\usepackage{amsthm,amsfonts,latexsym,amssymb,amscd} 	
\usepackage[mathscr]{eucal}   						
\usepackage[all]{xy}									
\usepackage[draft=false]{hyperref} 									

\newcommand{\E}{\mathcal{E}}\newcommand{\F}{\mathcal{F}}\newcommand{\G}{\mathcal{G}}
\newcommand{\I}{\mathcal{I}} 

\renewcommand{\O}{\mathcal{O}}	
 
\newcommand{\X}{\mathcal{X}}

\newcommand{\Eg}{\mathfrak{E}}\newcommand{\Fg}{\mathfrak{F}}\newcommand{\Gg}{\mathfrak{G}}

\newcommand{\CC}{{\mathbb{C}}}

\newcommand{\TT}{{\mathbb{T}}}

\newcommand{\As}{{\mathscr{A}}}\newcommand{\Bs}{{\mathscr{B}}}\newcommand{\Cs}{{\mathscr{C}}}
\newcommand{\Ds}{{\mathscr{D}}}
\newcommand{\Es}{{\mathscr{E}}}
\newcommand{\Hs}{{\mathscr{H}}}
\newcommand{\Is}{{\mathscr{I}}}\newcommand{\Js}{{\mathscr{J}}} 
\newcommand{\Ls}{{\mathscr{L}}}	
\newcommand{\Ms}{{\mathscr{M}}}\newcommand{\Ns}{{\mathscr{N}}}\newcommand{\Os}{{\mathscr{O}}}
\newcommand{\Rs}{{\mathscr{R}}}
\newcommand{\Ts}{\mathscr{T}} 			

\newcommand{\Xs}{{\mathscr{X}}}

\DeclareFontFamily{U}{rsfs}{\skewchar\font127 }
\DeclareFontShape{U}{rsfs}{m}{n}{%
   <5> <6> rsfs5
   <7> rsfs7
   <8> <9> <10> <10.95> <12> <14.4> <17.28> <20.74> <24.88> rsfs10
}{}
\DeclareSymbolFont{rsfs}{U}{rsfs}{m}{n} 
\DeclareSymbolFontAlphabet{\scr}{rsfs}

\newcommand{\Af}{\scr{A}}

\newcommand{\Tf}{\scr{T}}

\DeclareMathOperator{\ke}{Ker}

\DeclareMathOperator{\spa}{span} 

\DeclareMathOperator{\Ob}{Ob}

\DeclareMathOperator{\Hom}{Hom}

\DeclareMathOperator{\Aut}{Aut}
\DeclareMathOperator{\Sp}{Sp}
\DeclareMathOperator{\ev}{ev}

\DeclareMathOperator{\End}{End}

\DeclareMathOperator{\Rep}{Rep}

\renewcommand{\emph}{\textbf} 										
\newcommand{\cj}[1]{\overline{#1}}									
\newcommand{\ip}[2]{\langle #1\mid #2\rangle}					
\newcommand{\cs}{C*}

\newcommand{\hlink}[2]{\href{#1}{\texttt{#2}}} 

\newtheorem{theorem}{Theorem}[section]							

\newtheorem{lemma}[theorem]{Lemma}
\newtheorem{proposition}[theorem]{Proposition}
 
\newtheorem{definition}[theorem]{Definition}
\newtheorem{remark}[theorem]{Remark}

\numberwithin{equation}{section}  		
\title{\textbf{A Horizontal Categorification of Gel'fand Duality}}

\author{\normalsize  
Paolo Bertozzini \footnote{Partially supported by the Thai Research Fund: grant n.~RSA4780022.} @, 
Roberto Conti $^*\ddag$, 
Wicharn Lewkeeratiyutkul $^*\S$ 
\\ 
\normalsize @ e-mail: \texttt{paolo.th@gmail.com} 
\\
\normalsize $\ddag$ \textit{Mathematics, School of Mathematical and Physical Sciences,} 
\\ 
\normalsize \textit{University of Newcastle, Callaghan, NSW 2308, Australia}
\\ 
\normalsize  e-mail: \texttt{Roberto.Conti@newcastle.edu.au} 
\\
\normalsize $\S$ \textit{Department of Mathematics, Faculty of Science,}
\\
\normalsize \textit{Chulalongkorn University, Bangkok 10330, Thailand} 
\\ 
\normalsize  e-mail: \texttt{Wicharn.L@chula.ac.th}
}

\date{\normalsize{22 July 2010}}

\begin{document}

\maketitle

\begin{center} 
\textit{This paper is dedicated to J.~E.~Roberts, the ``pioneer'' of C*-categories.}
\end{center}

\begin{abstract} \noindent 
In the setting of C*-categories, we provide a definition of spectrum of a commutative full C*-category as a one-dimensional unital  Fell bundle over a suitable groupoid (equivalence relation) and prove a categorical Gel'fand duality theorem generalizing the usual Gel'fand duality between the categories of commutative unital \hbox{C*-algebras} and compact Hausdorff spaces. 
Although many of the individual ingredients that appear along the way are well known, 
the somehow unconventional way we ``glue'' them together seems to shed some new light on the subject.

\medskip
 
\noindent 
MSC-2000: 
					46L87,			
					46M15, 			
					46L08,			
					46M20, 			
					16D90,			
					18F99. 			

\medskip

\end{abstract}

\section{Introduction}

There is no need to explain why the notions of geometry and space are
fundamental both in mathematics and in physics. 
Typically, a rigorous way to encode at least some basic geometrical content  into a mathematical framework makes use of the notion of a topological space, i.e.~a set equipped with a topological structure.
Although being just a preliminary step in the process of developing a more sophisticated apparatus, this way of thinking has been very fruitful for both abstract and concrete purposes. 

In a very important development, I.~M.~Gel'fand looked not at the topological space itself but rather at the space of all continuous functions on it, and realized that these seemingly different structures are in fact essentially the same. 
In slightly more precise terms, he found a basic example of anti-equivalence between certain categories of spaces and algebras 
(see for example~\cite[Theorems~II.2.2.4, II.2.2.6]{Bl} or~\cite[Section~6]{L2}). 
Since on the analytic side $C(X;\CC)$ is a special type of a Banach algebra called a C*-algebra, the study of possibly 
non-commutative C*-algebras has been often regarded as a good framework for ``non-commutative topology''. 

The duality aspect has been later enforced by the Serre-Swan equivalence~\cite[Theorem~6.18]{Kar} between vector bundles and suitable modules (see also~\cite{FGV} for a Hermitian version of the theorem and~\cite{Ta1,Ta2,W} for generalizations involving Hilbert bundles). 
By then, breakthrough results have continued to emerge both in geometry and functional
analysis, based on Gel'fand's original intuition, for about four decades. 

In connection with physical ideas, L.~Crane-D.~Yetter~\cite{CY} and J.~Baez-J.~Dolan~\cite{BD} have recently proposed a process of categorification of mathematical structures, in which sets and functions are replaced by categories and functors. 

From this perspective, in this paper, we wish to discuss a categorification of the notion of space extending and merging together Gel'fand duality and Serre-Swan equivalence. 

On one side of the extended duality we have a horizontal categorification (a terminology that we introduced 
in~\cite[Section~4.2]{BCL0}) of the notion  of commutative C*-algebra, namely a  commutative C*-category, or 
commutative \hbox{C*-algebroid} (see definition~\ref{def: c*cat}),  whilst the corresponding replacement of spaces, 
the spaceoids (see definition~\ref{def: spaceoid}), are supposed to parametrize their spectra. 
Spaceoids could be described in several different albeit equivalent ways. 
In this paper we have decided to focus on a characterization based on the notion of Fell bundle. Originally Fell bundles were introduced in connection with the study of representations of locally compact groups, but we argue that they come to life naturally on the basis of purely topological principles. 

Rather surprisingly, to the best of our knowledge, the notions of commutative C*-category and its spectrum have not been discussed before, despite the fact that (mostly highly non-commutative) C*-categories have been somehow intensively exploited over the last 30 years in several areas of research, including Mackey induction, superselection structure in quantum field theory, abstract group duality, subfactors and the Baum-Connes conjecture. 
At any rate, we make frequent contact with the related notions that can be found in the literature, hoping that our approach sheds new light on the subject by approaching the matter from a kind of unconventional viewpoint. 

Of course, once we have a running definition, it seems quite challenging in the next step to look for some natural occurrence of the notion of spaceoid in other contexts. 
For instance, we are not aware of any connection with the powerful concepts that have been introduced in algebraic topology to date.
Also, the appearance of bundles in the structure of the spectrum suggests an intriguing connection to local gauge theory but we have not developed these ideas yet. 
Some of our considerations have been motivated by a categorical approach to non-commutative geometry~\cite{BCL0}, 
and it is rewarding that some of its relevant tools (e.g., Serre-Swan theorem, Morita equivalence) appear naturally in our context.
More structure is expected to emerge when our categories are equipped with a differentiable structure. 
In the case of usual spaces, in the setting of A. Connes' non-commutative geometry~\cite{C}, this has been achieved by means of a Dirac operator, and then axiomatized using the concept of spectral triple. 

\medskip

Here below we present a short description of the content of the paper.

In section~\ref{sec: catA} we mention, mainly for the purpose of fixing our notation, some basic definitions on C*-categories. 
Section~\ref{sec: catT} opens recalling the notion of a Fell bundle in the case of involutive inverse base categories and then proceeds to introduce the definition of the category of spaceoids that will eventually subsume that of compact Hausdorff spaces in our duality theorem. 
The construction of a small commutative full C*-category starting from a spaceoid is undertaken in section~\ref{sec: funG}, while the spectral analysis of a commutative full C*-category is the subject of the more technical section~\ref{sec: funS} where a spectrum functor from the category of full commutative C*-categories to our category of spaceoids is defined.  

Section~\ref{sec: gelfand} presents the main result of this paper in the form of a duality between a certain category of commutative full C*-categories and the category of their spectra (spaceoids). 
A categorified version of Gel'fand transform is introduced and used to prove a Gel'fand spectral reconstruction theorem for full commutative C*-categories. Similarly a categorified evaluation transform is defined for the purpose of proving the representativity of the spectrum functor. 

Section~\ref{sec: exappl} is devoted to examples and applications. Here we mention several natural examples of commutative full C*-categories and we produce explicit constructions of spaceoids, either reassembling the Hermitian line bundles obtained 
in~\cite{BCL3} as spectra of imprimitivity Hilbert C*-bimodules, or (in a way completely independent from C*-categories) as associated line bundles for a categorified version of $\TT$-torsors. Among the several possible future applications of such categorified Gel'fand duality, we describe in some detail a categorified continuous functional calculus. 
The paper ends with an outlook. 

\medskip

While in the usual Gel'fand duality theory a spectrum is just a compact topological space, in the situation under consideration it comes up equipped with a natural bundle structure. 

In particular, the spectrum of a commutative full C*-category can be identified with a kind of ``groupoid of Hermitian line bundles"\footnote{To be more precise, the spectrum is uniquely associated to a connected groupoid (actually an equivalence relation) of Hermitian line bundles, over a given compact Hausdorff space $X$, with compositions and involutions provided by a strict realization of fiberwise tensor products and fiberwise duals (the spectrum being the $*$-category whose class of arrows is obtained as union of the $\Hom$ sets of this groupoid).} that can be conveniently described using the language of Fell bundles 
(or equivalently as a continuous field of one-dimensional C*-categories). 

Along the way, we also discuss several categorical versions of some well-known concepts like the Gel'fand transform that we think are of independent interest. 
Notice that a notion of Fourier transform in the setting of compact groupoids has been discussed by M.~Amini~\cite{Am}.

Our duality is reminiscent of an interesting but widely ignored duality result of A.~Takahashi~\cite{Ta1,Ta2}. 
Takahashi's duality can be essentially understood as a duality of 
categories equipped with a partially defined weakly associative tensor product and a (weak) involution,  
although he does not explicitly examine such natural structures on his categories of Hilbert bundles and Hilbert C*-modules.\footnote{With more precision, Takahashi's result can be formally expressed as a duality between (a weak involutive form of) 
``2-fold-categories'' (also called double categories in the literature).} 
The duality presented in this paper is essentially a version of the former, where we exploit a strict realization of these tensor products and involutions, considering commutative full C*-categories and spaceoids instead of Hilbert C*-modules and Hilbert bundles. 

\medskip 

Most of the results presented here have been announced in our survey paper~\cite{BCL0} and have been presented in several seminars in Thailand, Australia, Italy, UK since May 2006. 

\medskip

\emph{Note added in proof.} 
When the present work was under preparation, we became aware of some related results in T.~Timmermann's 
PhD dissertation~\cite{Ti} where, in the context of Hopf algebraic quantum groupoids, a very general non-commutative Pontryagin duality theory is developed by means of pseudomultiplicative unitaries in \cs-modules; and also in 
V.~Deaconu-A.~Kumjian-B.~Ramazan~\cite{DKR}, where a notion of Abelian Fell bundle  (which contains our commutative 
\cs-categories as a special case) is introduced and a structure theorem for them (in terms of ``twisted coverings of groupoids'') is proved. 
In the framework of $T$-duality, a Pontryagin type duality between commutative principal bundles and gerbes has been proposed by C.~Daenzer~\cite{D}; while a generalization of Pontryagin duality for locally compact Abelian group bundles has been provided by G.~Goehle~\cite{Go}. 

\section{Category $\Af$ of full commutative C*-categories}\label{sec: catA}

The notion of C*-category, introduced by J.~Roberts (see P.~Ghez-R.~Lima-J.~Roberts~\cite{GLR} and also P.~Mitchener~\cite{M1}) has been extensively used in algebraic quantum field theory: 
\begin{definition} \label{def: c*cat}
A \emph{C*-category} is a category $\Cs$ such that: the sets $\Cs_{AB}:=\Hom_\Cs(B,A)$ are complex Banach spaces; the compositions are bilinear maps such that $\|xy\|\leq\|x\|\cdot \|y\|$, for all $x\in \Cs_{AB}$, $y\in \Cs_{BC}$; there is an involutive antilinear contravariant functor $*:\Cs\to\Cs$, acting identically on the objects, such that $\|x^*x\|=\|x\|^2$, for all $x\in \Cs_{BA}$ and such that $x^*x$ is a positive element in the C*-algebra $\Cs_{AA}$, for every $x\in \Cs_{BA}$ (i.e.~$x^* x = y^* y$ for some 
$y \in \Cs_{AA}$). 
\end{definition}
In a C*-category $\Cs$, the ``diagonal blocks'' $\Cs_{AA}:=\Hom_\Cs(A,A)$ are unital C*-algebras and the ``off-diagonal blocks'' 
$\Cs_{AB}:=\Hom_\Cs(B,A)$ are unital Hilbert C*-bimodules on the \hbox{C*-algebras} $\Cs_{AA}$ and $\Cs_{BB}$. 
We say that $\Cs$ is \emph{full} if all the bimodules $\Cs_{AB}$ are imprimitivity bimodules. 
In practice, every full C*-category is uniquely associated to a ``strictification'' of a sub-equivalence relation of the Picard-Rieffel groupoid of unital C*-algebras,\footnote{By this we mean that, in a full C*-category $\Cs$, the family of $\Hom_\Cs(\cdot,\cdot)$ spaces is itself a strict subcategory of $1$-arrows in the weak $2$-C*-category of Hilbert C*-bimodules 
(with Rieffel tensor products and duals as compositions and involutions), that ``projects onto'' a sub-equivalence relation of the 
Picard-Rieffel groupoid (see~\cite{BCL3} for details on these categories).} where the term 
\emph{sub-equivalence relation} of a category denotes a subcategory that is itself an equivalence relation.
It is also very useful to see a C*-category as an involutive category fibered over the equivalence relation of its objects: in this way, a (full) C*-category becomes a special case of a (saturated) unital Fell bundle over an involutive (discrete) base category as described in 
definition~\ref{def: fell-bundle} below. 
We say that $\Cs$ is \emph{one-dimensional} if all the bimodules $\Cs_{AB}$ are one-dimensional and hence Hilbert spaces. 
Clearly, all one-dimensional C*-categories are full.

The first problem that we have to face is how to select a suitable full subcategory $\Af$ of ``commutative'' full C*-categories playing the role of horizontal categorification of the category of commutative unital C*-algebras. 
Since we are working in a completely strict categorical environment, our choice is to define a C*-category $\Cs$ to be 
\emph{commutative} if all its diagonal blocks $\Cs_{AA}$ are commutative C*-algebras. 

If $\Cs, \Ds\in \Af$ are two full commutative small C*-categories (with the same cardinality of the set of objects),
a \emph{morphism} in the category $\Af$ is an object bijective $*$-functor $\Phi:\Cs\to \Ds$. 

\medskip

For later usage, recall from~\cite[Definition~1.6]{GLR} and~\cite[Section~4]{M1} that a closed two-sided ideal $\Is$ in 
a C*-category $\Cs$ is always a $*$-ideal and that the quotient $\Cs/\Is$ has a natural structure as a C*-category 
with a natural \emph{quotient functor} $\pi: \Cs\to \Cs/\Is$. 
We have this ``first isomorphism theorem'', whose proof is standard. 
\begin{theorem}\label{th: fit}
Let $\Phi: \Cs \to \Ds$ be a $*$-functor between C*-categories. The \emph{kernel} of $\Phi$ defined by $\ker \Phi:=\{x\in \Cs\ | \ \Phi(x)=0\}$ is a closed two-sided ideal in $\Cs$ and there exists a unique $*$-functor $\check{\Phi}: \Cs/\ker \Phi\to \Ds$ such that 
$\check{\Phi} \circ \pi = \Phi$. 
The functor $\check{\Phi}$ is faithful and it is full if and only if $\Phi$ is full. 
\end{theorem}

Recall (see~\cite[Definition~1.8]{GLR}) that a \emph{representation} of a C*-category $\Cs$ is a $*$-functor $\Phi:\Cs\to\Hs$ with values in the C*-category $\Hs$ of bounded linear maps between Hilbert spaces. We end this section with a 
simple observation, whose proof is omitted. 

\begin{lemma}\label{cor: one-dim}
A one-dimensional C*-category $\Cs$, admits at least one $*$-functor $\gamma:\Cs\to\CC$.   
\end{lemma}

\section{Category $\Tf$ of full topological spaceoids}\label{sec: catT}

We now proceed to the identification of a good category $\Tf$ of ``spaceoids'' playing the role of horizontal categorification of the category of continuous maps between compact Hausdorff topological spaces. 
Making use of Gel'fand duality (see e.g.~\cite[Section~6]{L2}) for the diagonal blocks $\Cs_{AA}$ and  (Hermitian) Serre-Swan equivalence (see e.g.~\cite[Section~2.1.2]{BCL0} and references therein) 
for the off-diagonal blocks $\Cs_{AB}$ of a commutative full C*-category $\Cs$, we see that the spectrum of $\Cs$ identifies a 
sub-equivalence relation embedded in the Picard groupoid of Hermitian line bundles over the Gel'fand spectra of the diagonal 
C*-algebras $\Cs_{AA}$. 
Finally, reassembling such block-data, we recognize that, globally, the spectrum of a commutative full C*-category can be described as a very special kind of a Fell bundle that we call a full topological spaceoid. 
Fell bundles over topological groups were first introduced by J.~Fell~\cite[Section~II.16]{FD} and later generalized to the case of groupoids by S.~Yamagami (see A.~Kumjian~\cite{K} and references therein) and to the case of inverse semigroups by N.~Sieben 
(see R.~Exel~\cite[Section~2]{E}). 
These notions admit a natural extension to that of a Fell bundle over an involutive inverse category\footnote{By \emph{involutive category} we mean a category $\Xs$ equipped with an involution i.e.~an object preserving contravariant functor 
$*:\Xs\to\Xs$ such that $(x^*)^*=x$ for all $x\in \Xs$. If $\Xs$ has a topology we also require composition and involution to be continuous. $\Xs$ is an involutive \emph{inverse category} if $xx^*x=x$ for all $x\in \Xs$.} that we describe in definition~\ref{def: fell-bundle} below. 
For the definition of a Banach bundle, we refer to J.~Fell-R.~Doran~\cite[Section~I.13]{FD}. 
We recall, from A.~Kumjian~\cite[Section~2]{K}, that a Fell bundle over a groupoid is a Banach bundle $(\Es,\pi,\Xs)$ over a topological groupoid $\Xs$ whose total space $\Es$ is equipped with a continuous involution $*:\Es\to\Es$, denoted here by $*:e\mapsto e^*$, and with a continuous multiplication 
$\circ:\Es^2\to\Es$, denoted here by $\circ:(e_1,e_2)\mapsto e_1e_2$, defined on the set 
$\Es^2:=\{(e_1,e_2) \ | \ (\pi(e_1),\pi(e_2))\in \Xs^2\}$, where $\Xs^n$ denotes the family of $n$-paths in the groupoid $\Xs$, satisfying the following properties: 
\begin{itemize}
\item[] $\pi(e_1e_2)=\pi(e_1)\pi(e_2)$, for all $(e_1,e_2)\in \Es^2$, 
\item[] for all $(x_1,x_2)\in \Xs^2$, the multiplication map is bilinear when restricted to the sets $\Es_{x_1}\times\Es_{x_2}$, 
where $\Es_x:=\pi^{-1}(x)$, 
\item[] $(e_1e_2)e_3=e_1(e_2e_3)$, for all $(e_1,e_2,e_3)\in \Es$ such that $(\pi(e_1),\pi_2(e_2),\pi(e_3))\in \Xs^3$, 
\item[($1$)] $\|e_1e_2\|\leq \|e_1\|\cdot \|e_2\|$, for all $(e_1,e_2)\in\Es^2$,   
\item[] $\pi(e^*)=\pi(e)^*$, for all $e\in\Es$, where $\pi(e)^*$ denotes the inverse of $\pi(e)$ in $\Xs$, 
\item[] for all $x\in \Xs$, the involution map is conjugate-linear when restricted to the set $\Es_x$, 
\item[] $(e^*)^*=e$, for all $e\in \Es$, 
\item[] $(e_1e_2)^*=e_2^*e_1^*$, for all $(e_1,e_2)\in \Es^2$, 
\item[($2$)] $\|e^*e\|=\|e\|^2$, for all $e\in \Es$, 
\item[($3$)] the element $e^*e$ is positive in the C*-algebra $\Es_{\pi(e^*e)}$, for all $e\in\Es$.
\end{itemize}
This definition can be recasted in the following concise and slightly generalized form, that we
systematically adopt in the sequel. 
\begin{definition}\label{def: fell-bundle} 
A \emph{unital Fell bundle} $(\Es,\pi,\Xs)$ \emph{over an involutive inverse category} $\Xs$ is a Banach bundle that is also an involutive category $\Es$ fibered over the involutive category $\Xs$ with continuous fiberwise bilinear compositions and fiberwise conjugate-linear involutions satisfying properties $(1)$, $(2)$ and $(3)$ as above.

We say that the Fell bundle is \emph{rank-one} if $\Es_x$ is one-dimensional for all $x\in \Xs$.
\end{definition}

Note that a C*-category $\Cs$ can always be seen as a Fell bundle over the maximal equivalence relation $\Ob_\Cs\times\Ob_\Cs$ of its objects, with fibers $\Cs_{AB}$, for all $(A,B)\in \Ob_\Cs\times\Ob_\Cs$. Conversely, a Fell bundle whose base is such an equivalence relation can always be seen as a C*-category.

\begin{definition}\label{def: spaceoid}
A \emph{topological spaceoid} (or simply a spaceoid, for short) $(\Es,\pi,\Xs)$ is a unital rank-one Fell bundle over the product involutive topological category $\Xs:=\Delta_X\times\Rs_\Os$ where $\Delta_X:=\{(p,p) \ | \ p\in X\}$ is the minimal equivalence relation of a compact Hausdorff space $X$ and $\Rs_\Os:=\Os\times\Os$ is the maximal equivalence relation of a discrete space 
$\Os$. 
\end{definition}

With a slight abuse of notation, 
the arrows of the base involutive category $\Xs$ of a full spaceoid will simply be 
denoted by $p_{AB}:=((p,p),(A,B))\in \Delta_X\times\Rs_\Os$. 

Note that, since a constant finite-rank Banach bundle over a locally compact 
Hausdorff space is locally trivial~\cite[Remark~I.13.9]{FD} and hence a vector bundle, 
a topological spaceoid is a Hermitian line bundle over $\Xs$ and is a disjoint 
union of the Hermitian line bundles $(\Es_{AB},\pi|_{\Es_{AB}},\Xs_{AB})$, with $\Xs_{AB}:=\Delta_X\times\{AB\}$ and  
$\Es_{AB}:=\pi^{-1}(\Xs_{AB})$. 

Furthermore, a topological spaceoid can be seen as a one-dimensional C*-category that is a 
coproduct (in the category of small C*-categories) of the ``continuous field'' of the 
one-dimensional C*-categories $\Es_p:=\pi^{-1}(\Xs_p)$, where $\Xs_p:=\{(p,p)\}\times \Rs_\Os$, for all $p\in X$. 

\medskip 

A \emph{morphism of spaceoids}\footnote{Morphisms of spaceoids can be seen as examples of J.~Baez notion of spans (in this case, a span of the Fell bundles of the spaceoids).} $(f,\F)\colon(\Es_1,\pi_1,\Xs_1)\to(\Es_2,\pi_2,\Xs_2)$ is a pair $(f,\F)$ where:
\begin{itemize} 
\item
$f:=(f_\Delta,f_\Rs)$ with $f_\Delta:\Delta_1\to\Delta_2$ being a continuous map of topological spaces and 
$f_\Rs:\Rs_1\to\Rs_2$ an isomorphism of equivalence relations;
\item
$\F:f^\bullet(\Es_2)\to\Es_1$ is a fiberwise linear continuous $*$-functor such that $\pi_1\circ \F=\pi_2^f$, where 
$(f^\bullet(\Es_2),\pi_2^f,\Xs_1)$ denotes the standard $f$-pull-back\footnote{Recall that 
$f^\bullet(\Es_2):=\{(p_{AB},e)\in\Xs_1\times\Es_2\ \mid \ f(p_{AB})=\pi_2(e)\}$ 
with $f\circ\pi_2^f=\pi_2\circ f^{\pi_2}$, 
where $\pi_2^f$ and $f^{\pi_2}$ are defined on $f^\bullet(\Es_2)$ by 
$\pi_2^f(p_{AB},e):=p_{AB}$ and $f^{\pi_2}(p_{AB},e):=e$. 
If $\Es_2$ is a Fell bundle over $\Xs_2$, $f^\bullet(\Es_2)$ is a Fell bundle over $\Xs_1$.} of $(\E_2,\pi_2,\Xs_2)$.
\end{itemize}  

Topological spaceoids constitute a category if compositions and identities are given by 
\begin{equation*}
(g,\G)\circ(f,\F):=(g\circ f, \F\circ f^\bullet(\G)\circ \Theta^{\Es_3}_{g,f}) \qquad
\text{and} \qquad \iota(\Es,\pi,\Xs):=(\iota_{\Xs},\iota_{\Xs}^\pi),  
\end{equation*} 
where $\Theta^{\Es_3}_{g,f}:(g\circ f)^\bullet(\Es_3)\to f^\bullet(g^\bullet(\Es_3))$ is the natural isomorphisms between standard pull-backs 
given by $\Theta^{\Es_3}_{g,f}(x_1,e_3):=(x_1,(f(x_1),e_3))$, for all 
$(x_1,e_3)\in (g\circ f)^\bullet(\Es_3)$, 
and $f^\bullet(\G)$, thanks to the functorial properties of pull-backs, is defined on the standard pull-back by 
$f^\bullet(\G)(x_1,(x_2,e_3)):=(x_1,\G(x_2,e_3))$, for all $(x_1,(x_2,e_3)))\in f^\bullet(g^\bullet(\Es_3))$. 

\section{The section functor $\Gamma$}\label{sec: funG}

Here we are going to define a \emph{section functor} $\Gamma:\Tf\to \Af$ that to every spaceoid $(\Es,\pi,\Xs)$, with 
$\Xs:=\Delta_X\times\Rs_\Os$, associates a commutative full C*-category $\Gamma(\Es)$ as follows:
\begin{itemize}
\item 
$\Ob_{\Gamma(\Es)}:=\Os$; 
\item 
for all $A,B\in \Ob_{\Gamma(\Es)}$, $\Hom_{\Gamma(\Es)}(B,A):=\Gamma(\Xs_{AB};\Es)$, 
where $\Gamma(\Xs_{AB};\Es)$ denotes the set of continuous sections 
$\sigma:\Delta_X\times\{(A,B)\}\to\Es$,  
$\sigma:p_{AB}\mapsto \sigma^{AB}_p\in\Es_{p_{AB}}$ 
of the restriction $(\Es_{AB},\pi|_{\Es_{AB}},\Xs_{AB})$ of $(\Es,\pi,\Xs)$ to the base space $\Xs_{AB}$. 
\item 
for all $\sigma\in \Hom_{\Gamma(\Es)}(B,A)$ and $\rho\in \Hom_{\Gamma(\Es)}(C,B)$: 
\begin{gather*}
\sigma\circ\rho:p_{AC} \mapsto (\sigma\circ\rho)^{AC}_p:=\sigma^{AB}_p\circ\rho^{BC}_p,  
\\
\sigma^*: p_{BA} \mapsto (\sigma^*)^{BA}_p:=(\sigma^{AB}_p)^*, 
\\  
\|\sigma\|:=\sup_{p\in \Delta_X}\|\sigma^{AB}_p\|_{\Es},  
\end{gather*}
with operations taken in the total space $\Es$ of the Fell bundle. 
\end{itemize}
In the following, since for all $\sigma\in \Gamma(\Es)=\biguplus_{AB}\Gamma(\Es)_{AB}$, the discrete indeces $AB$ are already implicit in the specification of the section $\sigma\in \Gamma(\Es)_{AB}$, we will simply use the shorter notation 
$\sigma_p:=\sigma_p^{AB}$ to denote the evaluation of the section $\sigma$ at the point $p_{AB}\in \Xs$. 

\medskip 

By construction, the commutative C*-category $\Gamma(\Es)$ so obtained has 
sections of Hermitian line bundles on a compact Hausdorff space 
as $\Hom$ spaces, and thus it is full.  

We extend now the definition of $\Gamma$ to the morphisms of $\Tf$. 
Let $(f,\, \F)$ be a morphism in $\Tf$ from $(\Es_1,\pi_1,\Xs_1)$ to  $(\Es_2,\pi_2,\Xs_2)$. 
Given $\sigma\in \Gamma(\Es_2)$, we consider the unique section $f^\bullet(\sigma):\Xs_1\to f^\bullet(\Es_2)$ 
such that $f^{\pi_2}\circ f^\bullet(\sigma)=\sigma\circ f$ and the composition $\F\circ f^\bullet(\sigma)$. 
In this way we have a map  
\begin{equation*}
\Gamma_{(f,\,\F)}: \Gamma(\Es_2)\to\Gamma(\Es_1), \quad \Gamma_{(f,\,\F)}:\sigma\mapsto\F\circ f^\bullet(\sigma), \quad \forall \sigma\in \Gamma(\Es_2).
\end{equation*} 

\begin{proposition}
For any morphism $(\Es_1,\pi_1,\Xs_1)\xrightarrow{(f,\,\F)} (\Es_2,\pi_2,\Xs_2)$ in the category $\Tf$, the map $\Gamma_{(f,\,\F)}:\Gamma(\Es_2)\to \Gamma(\Es_1)$ is a morphism in the category $\Af$. 

The pair of maps $\Gamma:(\Es,\pi,\Xs) \mapsto \Gamma(\Es)$ and $\Gamma:(f,\,\F)\mapsto\Gamma_{(f,\,\F)}$ gives a contravariant functor from the category $\Tf$ of spaceoids to the category $\Af$ of small full commutative C*-categories. 
\end{proposition}
\begin{proof}
Let $(\Es_1,\pi_1,\Xs_1)\xrightarrow{(f,\,\F)}(\Es_2,\pi_2,\Xs_2)$ and 
$(\Es_2,\pi_2,\Xs_2)\xrightarrow{(g,\,\G)}(\Es_3,\pi_3,\Xs_3)$ be two composable morphisms in the category $\Tf$ and let 
$(\Es,\pi,\Xs)\xrightarrow{(\iota_{\Xs},\,\iota_\Xs^\pi)}(\Es,\pi,\Xs)$ be the identity morphism of $(\Es,\pi,\Xs)$. 
To complete the proof we must show that  
\begin{equation*}
\Gamma_{(g,\,\G)\circ(f,\,\F)}=\Gamma_{(f,\,\F)}\circ\Gamma_{(g,\,\G)}, \quad 
\Gamma_{(\iota_{\Xs},\,\iota_\Xs^\pi)}=\iota_{\Gamma(\Es)}, 
\end{equation*}
and these are obtained by tedious but straightforward calculations. 
\end{proof}

\section{The spectrum functor $\Sigma$}\label{sec: funS}

This section is devoted to the construction of a \emph{spectrum functor} $\Sigma:\Af\to\Tf$ that to every commutative full 
C*-category $\Cs$ associates its \emph{spectral spaceoid} $\Sigma(\Cs)$. 

\medskip 

Letting $\Cs$ be a C*-category, we denote by $\Rs^\Cs$ the topologically discrete 
$*$-category $\Cs/\Cs\simeq\Rs_{\Ob_\Cs}$ and by $\CC\Rs^\Cs:=\rho^\bullet(\CC)$ 
the one-dimensional C*-category pull-back of $\CC$ (considered as a 
\hbox{C*-category} with only one object $\bullet$) under the constant map 
$\rho:\Rs^\Cs\to \{\bullet\}$. Note that, via the canonical projection $*$-functor $\Cs\to\Rs^\Cs:=\Cs/\Cs$, from the defining property of pull-backs there is a bijective map $\omega\mapsto \widetilde{\omega}$ between the set of 
$\CC$-valued $*$-functors $[\Cs;\CC]$ and the object-preserving elements in the set of $\CC\Rs^\Cs$-valued $*$-functors 
$[\Cs;\CC\Rs^\Cs]$.

By definition, two $*$-functors $\omega_1$, $\omega_2$ in $[\Cs;\CC]$ are \emph{unitarily equivalent}, see 
P.~Mitchener~\cite[Section~2]{M2}, if there exists a ``unitary'' natural transformation $A\mapsto\nu_A\in \TT$ between them. 
\\ 
Note that the set $\Is_\omega:=\{x\in \Cs\ | \ \omega(x)=0\}$, which is also equal to $\{x\in \Cs\ | \ \omega(x^*x)=0\}$, is an ideal in 
$\Cs$ and $\Is_{\omega_1}=\Is_{\omega_2}$ if (and only if) the equivalence classes $[\omega_1]$ and $[\omega_2]$ coincide. 

We also need the following lemmas whose routine proofs are omitted:\footnote{Note that, for $\omega\in[\Cs;\CC]$ and 
$A,B\in \Ob_\Cs$, we denote by $\omega_{AB}$ the restriction of $\omega$ to  $\Cs_{AB}$.} 
\begin{lemma}\label{lem: principal}
If $\omega,\omega'\in [\Cs;\CC]$ are unitarily equivalent, there is a unique map $\psi:\Rs^\Cs\to\TT$ such that 
$\omega'_{AB}=\psi_{AB}\cdot\omega_{AB}$ for all $AB\in \Rs^\Cs$ 
given by $\psi_{AB} = \nu_B \nu_A^{-1}$, where $\nu$ is a unitary natural 
transformation from $\omega$ to $\omega'$,
and the map $\psi:AB\mapsto\psi_{AB}$ is a $*$-functor:
\begin{equation}\label{eq: fasi}
\psi_{AB}\psi_{BC}=\psi_{AC}, \quad 
\psi_{AB}=\psi_{BA}^{-1},  \quad  
\psi_{AA}=1_\CC.
\end{equation}
Conversely, given a $*$-functor
$\psi\in[\Rs^\Cs;\TT]$, two $*$-functors $\omega,\omega'$ such that $\omega'_{AB}=\psi_{AB}\omega_{AB}$ are unitarily equivalent. 
\end{lemma}

\begin{lemma}\label{lem: iso-cr}
Every object preserving $*$-automorphism $\gamma$ of the C*-category $\CC\Rs^\Cs$ is given by the multiplication by an element 
$\psi\in[\Rs^\Cs;\TT]$ i.e.~$\gamma(x)=\psi_{AB}\cdot x$ for all $x\in (\CC\Rs^\Cs)_{AB}$. 
\end{lemma}

\begin{proposition}\label{prop: field} 
Two $*$-functors $\omega$, $\omega' \in [\Cs;\CC]$ are unitarily equivalent if and only if $\omega_{AA}=\omega'_{AA}$ for all 
$A\in \Ob_\Cs$. 
\end{proposition}
\begin{proof}
By lemma~\ref{lem: principal}, if $[\omega]=[\omega']$, then $\omega'_{AA}=\psi_{AA}\cdot\omega_{AA}=\omega_{AA}$, for all objects $A$. 

Let $\omega,\omega'\in[\Cs;\CC]$ and suppose that $\omega_{AA}=\omega'_{AA}$, for all $A\in \Ob_\Cs$. Consider the corresponding object-preserving 
$\CC\Rs^\Cs$-valued $*$-functors $\widetilde{\omega}$, $\widetilde{\omega}'\in[\Cs;\CC\Rs^\Cs]$. 
Note that $\ke(\widetilde{\omega})=\Is_\omega=\Is_{\omega'}=\ke(\widetilde{\omega}')$ and hence, $\omega_{AB}$,  
$\widetilde{\omega}_{AB}$ are nonzero if and only if $\omega'_{AB}$, $\widetilde{\omega}'_{AB}$ are nonzero. 
If $\omega_{AB}$ is nonzero for all $AB\in\Rs^\Cs$, by theorem~\ref{th: fit} we have two $*$-isomorphisms 
\begin{equation*}
\Cs/\ke(\omega)\xrightarrow{\hat{\omega}}\CC\Rs^\Cs\xleftarrow{\,\hat{\omega}'}\Cs/\ke(\omega'). 
\end{equation*} 
From lemma~\ref{lem: iso-cr} there is a 
$\psi\in[\Rs^\Cs;\TT]$ such that 
$\hat{\omega}'=\psi\cdot \hat{\omega}$ and hence also $\omega'=\psi\cdot\omega$ so that the proposition follows from 
lemma~\ref{lem: principal}. 

In order to complete the proof, notice that the images of $\hat\omega$ and 
$\hat\omega'$ coincide and then apply the above argument to the connected components of the image category. 
\end{proof}

Given $x\in \Cs_{AA}$ for some object $A$, 
by evaluation in $x$ we mean the map $\ev_x:[\Cs;\CC]\to\CC$ defined by $\omega\mapsto \omega(x)$.

\begin{proposition}
The set $[\Cs;\CC]$ of $\CC$-valued $*$-functors $\omega:\Cs\to\CC$, with the weakest topology making all evaluations continuous, is a compact Hausdorff topological space.  
\end{proposition}
\begin{proof}
Note that for all $\omega\in [\Cs;\CC]$ and for all $x\in \Cs_{AB}$, 
\begin{align*}
|\omega(x)|&=\sqrt{\cj{\omega(x)}\omega(x)}=\sqrt{\omega(x^*x)} =\sqrt{\omega_{AA}(x^*x)}
\leq\sqrt{\|x^*x\|}=\sqrt{\|x\|^2}=\|x\|, 
\end{align*}
because $\omega_{AA}$ is a state over the C*-algebra $\Cs_{AA}$. Hence $[\Cs;\CC]$ is a subspace of the 
compact Hausdorff space $\prod_{x\in \Cs}D_{\|x\|}$, where $D_{\|x\|}$ is the 
closed ball in $\CC$ of radius $\|x\|$, and it is easy to check that it is closed.
\end{proof}

Let $\Sp_b(\Cs):=\{[\omega] \ | \ \omega\in [\Cs;\CC]\}$ denote the \emph{base spectrum} of $\Cs$, 
defined as the set of unitary equivalence classes of $*$-functors in $[\Cs;\CC]$. 
It is a compact space with the quotient topology induced by the map  $\omega\mapsto[\omega]$. 
To show that $\Sp_b(\Cs)$ is Hausdorff it is enough to note that, by proposition~\ref{prop: field}, if $[\omega]\neq[\omega']$, there exists at least one object $A$ such that $\omega_{AA}\neq\omega'_{AA}$ and so there exists at least one evaluation $\ev_x$ with $x\in \Cs_{AA}$ such that $\ev_x(\omega)\neq\ev_x(\omega')$. Since, for $x\in\Cs_{AA}$, $\ev_x$ induces a well-defined map on the quotient space $\Sp_b(\Cs)$ by $[\omega]\mapsto\omega(x)$, the result follows. 

\begin{proposition}\label{prop: bije}
Let $\Cs$ be a full commutative C*-category. 
For all $A\in \Ob_\Cs$, there exists a natural bijective map, between the base spectrum of $\Cs$ and the usual Gel'fand spectrum 
$\Sp(\Cs_{AA})$ of the 
C*-algebra $\Cs_{AA}$, given by the restriction $|_{AA}:\omega\mapsto\omega|_{\Cs_{AA}}$. 

In particular, for all objects $A\in \Ob_\Cs$, one has 
$\Sp_b(\Cs)|_{AA}=\Sp (\Cs_{AA}) \simeq\Sp_b(\Cs_{AA})$.
\end{proposition}
\begin{proof}
By proposition~\ref{prop: field}, the correspondence $[\omega]\mapsto \omega_{AA}$ is well defined.

We show that the map $[\omega]\mapsto \omega_{AA}$ is injective. 
Given $\omega,\omega'\in [\Cs;\CC]$ with $\omega_{AA}=\omega'_{AA}$, we know 
from~\cite[Proposition~2.30]{BCL3}, that $\omega_{BB}(x)=\omega_{AA}(\phi_{AB}(x))$, for all $x\in \Cs_{BB}$, for all 
$B\in \Ob_\Cs$, where 
$\phi_{AB}: \Cs_{BB}\to\Cs_{AA}$ is the canonical isomorphism associated to the imprimitivity bimodule $\Cs_{BA}$. 
It follows that $\omega_{BB}=\omega_{AA}\circ\phi_{AB}=\omega'_{AA}\circ\phi_{AB}=\omega'_{BB}$, for all $B\in \Ob_\Cs$ and, by proposition~\ref{prop: field}, we see that $[\omega]=[\omega']$. 

\medskip

We show that the function $[\omega]\mapsto \omega_{AA}$ is surjective. 
Let $\omega^o\in \Sp(\Cs_{AA})$.
Define $\Js$ to be the involutive ideal in $\Cs$ generated by $\ke(\omega^o)$. 
One can see that $\Js \cap \Cs_{AA} = \Js_{AA} = \ke(\omega^o)$.
Since the quotient $*$-functor $\pi: \Cs \to \Cs/\Js$ is bijective on the objects 
and $\Cs$ is full, $\Cs/\Js$ is full. 
Since $\Cs_{AA}/\Js_{AA}$ is one-dimensional, the quotient C*-category $\Cs/\Js$ 
is one-dimensional. 
If $\gamma:\Cs/\Js\to \CC$ is a $\CC$-valued $*$-functor 
as in  lemma~\ref{cor: one-dim}, $\gamma \circ \pi$ restricted to $\Cs_{AA}$
must be $\omega^o$ since it vanishes on $\Js \cap \Cs_{AA}$. 
\end{proof}

\begin{theorem}\label{prop: homo-sp}
Let $\Cs$ be a full commutative C*-category.
For every $A\in \Ob_\Cs$, the bijective map $|_{AA}:\Sp_b(\Cs)\to\Sp(\Cs_{AA})$ given by $[\omega]\mapsto \omega_{AA}$ is a homeomorphism between $\Sp_b(\Cs)$ and the Gel'fand spectrum $\Sp(\Cs_{AA})$ of the unital C*-algebra $\Cs_{AA}$. 
\end{theorem}
\begin{proof} 
Since both $\Sp_b(\Cs)$ and $\Sp(\Cs_{AA})$ are compact Hausdorff spaces, and the map $\mid_{AA}$ is bijective, it is enough to show that $\mid_{AA}: \Sp_b(\Cs)\to\Sp(\Cs_{AA})$ is continuous. Since $\Sp_b(\Cs)$ is equipped with the quotient topology induced by the projection map $\pi: [\Cs;\CC]\to\Sp_b(\Cs)$, the map $\mid_{AA}$ is continuous if and only if $\mid_{AA}\circ \pi:[\Cs;\CC]\to\Sp(\Cs_{AA})$ is continuous. 
The spaces $[\Cs;\CC]$ and $\Sp(\Cs_{AA})$ are equipped with the weakest topology making the evaluation maps continuous. 
It follows that the continuity of $\mid_{AA}\circ\pi$ is equivalent to the continuity of 
$\ev_x = \ev_x\circ\mid_{AA}\circ\pi:[\Cs;\CC]\to\CC$ for all $x\in \Cs_{AA}$. 
Since $\ev_x:[\Cs;\CC]\to\CC$ is continuous, the result is established.   
\end{proof}

Let $\Xs^\Cs:=\Delta^{\Cs}\times\Rs^{\Cs}$ be the direct product equivalence relation of the compact Hausdorff $*$-category 
$\Delta^\Cs:=\Delta_{\Sp_b(\Cs)}$ and the topologically discrete $*$-category $\Rs^\Cs:=\Cs/\Cs\simeq\Rs_{\Ob_\Cs}$.

With a slight abuse of notation, we write $AB\in \Rs^\Cs$ for the arrow $\Cs_{AB}/\Cs_{AB}$ in $\Rs^\Cs$ and denote  
$([\omega], AB)=([\omega],\Cs_{AB}/\Cs_{AB})\in \Xs^\Cs$ simply by $\omega_{AB}$. 

We define $\Es^\Cs$ as the disjoint union over $[\omega]\in \Delta^\Cs$ of the quotients $\Cs/\Is_\omega$. In formulae: 
\begin{gather*} 
\Es^\Cs_\omega:=\frac{\Cs}{\Is_\omega}, \quad 
\Es^\Cs:=\biguplus_{\omega\in\Delta^\Cs} \Es^\Cs_\omega=\biguplus_{\omega_{AB}\in\Xs^\Cs}\Es^\Cs_{\omega_{AB}}, 
\\ 
\pi^\Cs \colon \Es^\Cs\to \Xs^\Cs, \qquad \pi^\Cs : e\mapsto \omega_{AB}, \quad \forall e\in \Es^\Cs_{\omega_{AB}},  
\end{gather*} 
where $\Es^\Cs_{\omega_{AB}}:=\Cs_{AB}/\Is_{\omega_{AB}}$, with $\Is_{\omega_{AB}}:=\Is_\omega\cap\Cs_{AB}$. 

\begin{proposition}\label{prop: egamma}
The triple $(\Es^\Cs,\pi^\Cs,\Xs^\Cs)$ is naturally equipped with the structure of a unital rank-one Fell bundle over the topological involutive inverse category $\Xs^\Cs$. 
\end{proposition}
\begin{proof}
Define on $\Es^\Cs$ the topology whose fundamental system of neighborhoods are the sets 
$U_{e_0}^{O,\, x_0,\, \varepsilon}:=\{e\in\Es^\Cs  \mid \ \pi^\Cs(e)\in O, \ \exists x\in\Cs : \hat{x}(\pi^\Cs(e))=e, \; \|\hat{x}-\hat{x}_0\|_\infty <\varepsilon\}$, where $e_0\in\Es^\Cs$, $O$ is open in $\Xs^\Cs$, $\varepsilon>0$, $x_0\in \Cs$ with 
$\hat{x}_0(\pi^\Cs(e_0))=e_0$ and where $\hat{x}$ denotes the Gel'fand transform of $x$ defined in section~\ref{sec: gel}. 
This topology entails that a net $(e_\mu)$ is convergent to the point $e$ in $\Es^\Cs$ if and 
only if the net $\pi^\Cs(e_\mu)$ converges to $\pi^\Cs(e)$ in $\Xs^\Cs$ and, 
for all possible Gel'fand transforms $\hat{x}_0$ ``passing'' through $e_0$, 
there exists a net of Gel'fand transforms $\hat{x}_\mu$, ``passing'' through $e_\mu$, 
that uniformly converges, on every neighborhood of $\pi^\Cs(e_0)$, to $\hat{x}_0$. 

With such a topology the (partial) operations on $\Es^\Cs$ i.e.~sum, scalar multiplication, product, involution, inner product (and hence norm) become continuous and $(\Es^\Cs,\pi^\Cs,\Xs^\Cs)$ becomes a Banach bundle. 

Since every sub-equivalence relation of $\Xs^\Cs$ is a disjoint union of ``grids'' $\{[\omega]\}\times\Rs^\Cs$ whose inverse image under $\pi^\Cs$ is the one-dimensional C*-category $\Cs/\ke(\omega)$, $(\Es^\Cs,\pi^\Cs,\Xs^\Cs)$ is a rank-one unital Fell bundle over the equivalence relation $\Xs^\Cs$ and hence a spaceoid.
\end{proof}

To a commutative full C*-category $\Cs$ we have associated a topological \emph{spectral spaceoid} 
$\Sigma(\Cs) := (\Es^\Cs,\pi^\Cs,\Xs^\Cs)$. 
We extend now the definition of $\Sigma$ to the morphisms of $\Af$. 
Let $\Phi:\Cs\to\Ds$ be an object-bijective $*$-functor between two small commutative full C*-categories with spaceoids $\Sigma(\Cs),\Sigma(\Ds)\in \Tf$ and define a morphism $\Sigma^\Phi:\Sigma(\Ds)\xrightarrow{(\lambda^\Phi,\Lambda^\Phi)}\Sigma(\Cs)$ in the category $\Tf$ as follows.  

\medskip

We set
$\lambda^\Phi:\Xs^\Ds\xrightarrow{(\lambda^\Phi_\Delta,\lambda^\Phi_\Rs)}\Xs^\Cs$ where $\lambda^\Phi_\Rs:\Rs^\Ds\to\Rs^\Cs$ is the isomorphism of equivalence relations given by $\lambda^\Phi_\Rs(AB):=\Phi^{-1}(A)\Phi^{-1}(B)$, for $AB \in \Rs^\Ds$, and where $\lambda^\Phi_\Delta:\Delta^\Ds\to\Delta^\Cs$ (since $\omega\mapsto \omega\circ \Phi$ is continuous and preserves equivalence by unitary natural transformations) is the well-defined continuous map given by 
$\lambda^\Phi_\Delta([\omega]):=[\omega\circ\Phi] \in \Delta^\Cs$, for all $[\omega]\in \Delta^\Ds$. 

\medskip

Since, for $\omega\in [\Ds;\CC]$, the $*$-functor $\Phi:\Cs\to\Ds$ induces a continuous field of quotient $*$-functors  
$\Phi_\omega:\Cs/\Is_{\Phi\circ\omega}\to \Ds/\Is_\omega$ between one-dimensional C*-categories, we can define\footnote{Note that $(\lambda^\Phi)^\bullet(\Es^\Cs)$ is the disjoint union of the continuous field of one-dimensional C*-categories $\Cs/\Is_{\Phi\circ\omega}$.} $\Lambda^\Phi: (\lambda^\Phi)^\bullet(\Es^\Cs)\to \Es^\Ds$ as the disjoint union of the $*$-functors $\Phi_\omega$, for $[\omega]\in \Delta^\Ds$ and note that it is a continuous fiberwise linear 
$*$-functor. 

\begin{proposition}
For any morphism $\Cs\xrightarrow{\Phi}\Ds$ in $\Af$, the map $\Sigma(\Ds)\xrightarrow{\Sigma^\Phi}\Sigma(\Cs)$ is a morphism of spectral spaceoids.  
The pair of maps $\Sigma:\Cs\mapsto\Sigma(\Cs)$ and $\Sigma: \Phi\mapsto\Sigma^\Phi$ gives a contravariant functor $\Sigma:\Af\to\Tf$, from the category $\Af$ of object-bijective $*$-functors between small commutative full C*-categories to the category $\Tf$ of spaceoids. 
\end{proposition}
\begin{proof}
We have to prove that $\Sigma$ is antimultiplicative and preserves the identities.  

If $\Phi:\Cs_1\to\Cs_2$ and $\Psi:\Cs_2\to\Cs_3$ are two $*$-functors in $\Af$, by definition, 
\[ 
 \Sigma^{\Psi\circ\Phi}=(\lambda^{\Psi\circ\Phi},\Lambda^{\Psi\circ\Phi})=\Big(\lambda^\Phi\circ\lambda^\Psi, \  \Lambda^\Psi\circ(\lambda^\Psi)^\bullet(\Lambda^\Phi) \circ
\Theta_{\lambda^\Phi, \lambda^\Psi}^{\Es^{\Cs_1}} \Big) 
=(\lambda^\Phi,\Lambda^\Phi)\circ(\lambda^\Psi,\Lambda^\Psi)
=\Sigma^{\Phi}\circ\Sigma^{\Psi}. 
\]
Also, if $\iota_\Cs:\Cs\to\Cs$ is the identity functor of the C*-category $\Cs$, then the morphism \hbox{$\Sigma^{\iota_\Cs}=(\lambda^{\iota_\Cs},\Lambda^{\iota_\Cs})$} is the identity morphism of the spaceoid $\Sigma (\Cs)$. 
\end{proof} 

\section{Horizontal Categorification of Gel'fand Duality}\label{sec: gelfand}

\subsection{Gel'fand Transform}\label{sec: gel}
For a given C*-category $\Cs$ in $\Af$, we define a horizontally categorified version of \emph{Gel'fand transform} as 
$\Gg_\Cs: \Cs\to\Gamma(\Sigma(\Cs))$ given by $\Gg_\Cs: x\mapsto\hat{x}\quad \text{where} \quad \hat{x}_{[\omega]}:=x+\Is_{\omega_{AB}}$,  for all $x\in \Cs_{BA}$. 
Clearly $\Gg_\Cs:\Cs\to\Gamma(\Sigma(\Cs))$ is an object bijective $*$-functor. 

\begin{lemma}\label{lem: subcat} 
Let $\Cs$ be a commutative C*-category and $\Cs^o$ a subcategory of $\Cs$ which is a full C*-category such that 
$\Cs^o_{AA}=\Cs_{AA}$ for all $A\in \Ob_\Cs=\Ob_{\Cs^o}$. 
Then $\Cs^o_{AB}=\Cs_{AB}$ for all $A,B\in \Ob_\Cs$. 
\end{lemma}
\begin{proof}
By the fullness of the bimodule ${}_\As\Cs^o_\Bs$ there is a sequence of pairs $u_j,v_j\in {}_\As\Cs^o_\Bs$ such that 
$\iota_\Bs=\sum_{j=1}^\infty u_j^*v_j$. 
We have 
$x=x\iota_\Bs=x\sum_{j=1}^\infty u^*_jv_j=\sum_{j=1}^\infty (xu^*_j)v_j \in {}_\As\Cs^o_\Bs$ for all $x\in {}_\As\Cs_\Bs$, 
because $xu^*_j\in {}_\As\Cs_\As={}_\As\Cs^o_\As$ and so $(xu^*_j)v_j\in {}_\As\Cs^o_\Bs$ for all $j$.
\end{proof}

\begin{theorem}\label{th: rep} 
The Gel'fand transform $\Gg_\Cs:\Cs\to\Gamma(\Sigma(\Cs))$ of a commutative full C*-category $\Cs$ is a full faithful isometric 
$*$-functor.  
\end{theorem}
\begin{proof} 
The proof follows from the fact that the Gel'fand transform $\Gg_\Cs$, when restricted to any ``diagonal'' commutative unital 
C*-algebra $\Cs_{AA}$ can be ``naturally identified'' with the usual Gel'fand transform of $\Cs_{AA}$ via the homeomorphism 
$[\omega]\mapsto \omega|_{AA}$ (see proposition~\ref{prop: bije} and theorem~\ref{prop: homo-sp}).  

To prove the faithfulness of $\Gg_\Cs$, let $x\in \Cs_{BA}$ with $\widehat{x}=0$. 
We have $\widehat{x^*x}=\widehat{x}^*\widehat{x}=0$ so that the usual Gel'fand transform of $x^*x\in \Cs_{AA}$ is zero and, by Gel'fand isomorphism theorem applied to the C*-algebra $\Cs_{AA}$, we have $x^*x=0$ and hence $x=0$. 

The ``image'' $\Gg_\Cs(\Cs)$ of $\Gg_\Cs$ is a subcategory of the commutative full C*-category $\Gamma(\Sigma(\Cs))$ that is clearly a commutative full C*-category on its own. 
By lemma~\ref{lem: subcat}, the $*$-functor $\Gg_\Cs$ is full as long as $\Gg_\Cs(\Cs_{AA})=\Gamma(\Sigma(\Cs))_{AA}$, for all objects $A\in \Ob_\Cs$ and this follows again by the usual Gel'fand isomorphism theorem applied to the C*-algebra $\Cs_{AA}$.  

For the isometry of $\Gg_\Cs$ we note that for all $x\in \Cs_{BA}$, since the Gel'fand transform restricted to the C*-algebra 
$\Cs_{AA}$ is isometric, we have $\|\Gg_\Cs(x)\|^2=\|\widehat{x^*x}\|=\|x^*x\|=\|x\|^2$. 
\end{proof} 

\subsection{Evaluation Transform}

Given a topological spaceoid $(\Es,\pi,\Xs)$, we define a horizontally categorified version of \emph{evaluation transform}
$\Eg_\Es: (\Es,\pi,\Xs)\xrightarrow{(\eta^\Es,\Omega^\Es)}\Sigma(\Gamma(\Es))$ as follows: 
\begin{itemize}
\item 
$\eta^\Es_\Rs:\Rs_\Os\to\Rs^{\Gamma(\Es)}$ is the canonical isomorphism $\Rs_\Os=\Rs_{\Ob_{\Gamma(\Es)}}\simeq\Gamma(\Es)/\Gamma(\Es)$, explicitly:  
$\eta^\Es_\Rs(AB):=\Gamma(\Es)_{AB}/\Gamma(\Es)_{AB}, \quad \forall AB\in \Rs_\Os$ that is, according to the running notation, written as an identity map $\eta^\Es_\Rs(AB)=AB\in \Rs^{\Gamma(\Es)}$. 
\item 
$\eta^\Es_\Delta: \Delta_X \to \Delta^{\Gamma(\Es)}$ is given by 
$\eta^\Es_\Delta: p\mapsto [\gamma_p\circ\ev_p]\quad \forall p\in \Delta_X$, where the evaluation map 
$\ev_p: \Gamma(\Es)\to \uplus_{(AB)\in\Rs_\O} \ \Es_{p_{AB}}$ given by 
$\ev_p: \sigma\mapsto \sigma_p$ is a $*$-functor with values in a one-dimensional C*-category that 
determines\footnote{By lemma~\ref{cor: one-dim}, there is always a $\CC$-valued $*$-functor 
$\gamma_p:\Es_p\to\CC$ and by proposition~\ref{prop: field} any two compositions of $\ev_p$ with such $*$-functors are unitarily equivalent because they coincide on the diagonal C*-algebras $\Es_{p_{AA}}$.} a unique point  
$[\gamma_p\circ\ev_p]\in \Delta_{\Sp_b(\Gamma(\Es))}$. 
\end{itemize}

\begin{itemize}
\item
$\Omega^\Es \colon (\eta^\Es)^\bullet(\Es^{\Gamma(\Es)})\to \Es$\, is defined by 
$\Omega^\Es \colon (p_{AB}, \ \sigma+\Is_{\eta^\Es(p_{AB})})
\mapsto \sigma_p$, \,$\forall \sigma \in \Gamma(\Es)_{AB}$, $\forall p_{AB}\in\Xs$. 
\end{itemize}

In particular, with such definitions we can prove that the spectrum functor is representative: 
\begin{theorem} 
The evaluation transform $\Eg_\Es :(\Es,\pi,\Xs)\to\Sigma(\Gamma(\Es))$, for all spaceoids $(\Es,\pi,\Xs)$, is an isomorphism in the category of spaceoids.   
\end{theorem}
\begin{proof} 
Note that $(\Es_{AA},\pi,X)$ is naturally isomorphic to the trivial $\CC$-bundle over $X$ and thus there is an isomorphism of the C*-algebras $\Gamma(\Es)_{AA}$ and $C(X)$ that ``preserves'' evaluations. 

Clearly the map $\zeta_A: \Delta_X\to \Sp(\Gamma(\Es)_{AA})$ given by  
$\zeta_A(p):= |_{AA}\circ\eta^\Es_\Delta(p)=\gamma_{p_{AA}}\circ\ev_{p_{AA}}$, 
coincides with the usual Gel'fand evaluation homeomorphism for the diagonal C*-algebra $\Gamma(\Es)_{AA}$   
and hence, by proposition~\ref{prop: bije}, $\eta^\Es_\Delta=|_{AA}^{-1}\circ\zeta_A$ is also a homeomorphism. 

\medskip

For every element $e\in \Es$, we have $\pi(e)\in\Delta_X\times\Rs_\Os$ and, since a spaceoid is actually a vector bundle, it is always possible to find a section $\sigma\in \Gamma(\Es)$ such that $\sigma_{\pi(e)}=e$. For any such section we consider the element $\sigma+\Is_{\eta^\Es(\pi(e))}\in\Gamma(\Es)/\Is_{\eta^\Es(\pi(e))}=:\Es^{\Gamma(\Es)}_{\eta^\Es(\pi(e))}$ (note that the element does not depend on the choice of $\sigma\in\Gamma(\Es)$ such that $\sigma_{\pi(e)}=e$) and in this way we have a map 
$\Theta:\Es\to\Es^{\Gamma(\Es)}$ by $\Theta:e\mapsto\sigma+\Is_{\eta^\Es(\pi(e))}$. 
The map $\Theta$ uniquely induces a function 
$\Xi^\Es:\Es\to(\eta^\Es)^\bullet(\Es^{\Gamma(\Es)})$ with the standard $\eta^\Es$-pull-back of $\Es^{\Gamma(\Es)}$ given by 
$\Xi^\Es(e):=(\pi(e),\Theta(e))$. By direct computation the map $\Xi^\Es$ is an ``algebraic isomorphism'' of Fell 
bundles\footnote{By this we mean that $\Xi^\Es:\Es\to(\eta^\Es)^\bullet(\Es^{\Gamma(\Es)})$ is a fiber preserving map between bundles, over the same base space $\Xs$, that is also a bijective fiberwise linear $*$-functor between the total spaces.} whose inverse is $\Omega^\Es$.  
  
The continuity of $\Omega^\Es$ is equivalent to that of $\widetilde{\Omega}^\Es: \Es^{\Gamma(\Es)}\to \Es$, 
$\widetilde{\Omega}^\Es(\sigma+\Is_{\eta^\Es(p_{AB})}):=\sigma_p$, with $\sigma\in \Gamma(\Es)_{AB}$. Given a net $j\to\sigma^j+\Is_{\eta^\Es(p^j_{AB^j})}$ in $\Es^{\Gamma(\Es)}$ converging to the point $\sigma+\Is_{\eta^\Es(p_{AB})}$ in the topology defined in proposition~\ref{prop: egamma}, without loss of generality we can assume that $j\to\sigma^j$ is uniformly convergent to $\sigma$ in a neighborhood $U$ of $\eta^\Es(p_{AB})$. This means that, for all $\epsilon>0$,  
$\|\sigma^j([\omega]_{AB})-\sigma([\omega]_{AB})\|<\epsilon$ 
for $[\omega]_{AB}\in U$ for all sufficiently large $j$. 
Since $\Rs^{\Gamma(\Es)}$ is discrete, the net $AB^j$ is eventually equal to $AB$ and since $\eta^\Es$ is a homeomorphism, 
$p^j_{AB}$ eventually lies in any neighborhood of $p_{AB}$ and hence the net 
$\widetilde{\Omega}^\Es(\sigma^j+\Is_{\eta^\Es(p^j_{AB^j})})=(\sigma^j)_{p^j}$ converges to 
$\widetilde{\Omega}(\sigma+\Is_{\eta^\Es(p_{AB})})=\sigma_p$ in the Banach bundle topology of $\Es$. 

Since $\Omega^\Es$ is an isometry, it follows from~\cite[Proposition~13.17]{FD} that its inverse is continuous too and hence the evaluation transform $\Eg^\Es:=(\eta^\Es,\Omega^\Es)$ is an isomorphism of spaceoids.
\end{proof} 

\subsection{Duality}

\begin{theorem}\label{th: catgel}
The pair of functors $(\Gamma,\Sigma)$ provides a duality between the category $\Tf$ of morphisms between spaceoids 
that are object-bijective on the discrete part of the objects
and the category $\Af$ of 
object-bijective $*$-functors between small commutative full C*-categories. 
\end{theorem}
\begin{proof}
To see that the map $\Gg:\Cs\mapsto\Gg_\Cs$ (that to every $\Cs\in \Ob_\Af$ associates the Gel'fand transform of $\Cs$) is a natural isomorphism between the identity endofunctor $\I_\Af:\Af\to \Af$ and the functor $\Gamma\circ\Sigma:\Af\to \Af$ we have to show that, given an object-bijective $*$-functor $\Phi:\Cs_1\to\Cs_2$, the identity 
$\Gamma_{\Sigma^\Phi}(\Gg_{\Cs_1}(x))=\Gg_{\Cs_2}(\Phi(x))$ holds for any $x\in \Cs_1$, i.e.~the commutativity of the diagram: 
\begin{equation*} 
\xymatrix{
\Cs_1 \ar[rr]^{\Gg_{\Cs_1}} \ar[d]_{\Phi} & & \Gamma(\Sigma(\Cs_1))\ar[d]^{\Gamma_{\Sigma^\Phi}} \\
\Cs_2 \ar[rr]^{\Gg_{\Cs_2}} & & \Gamma(\Sigma(\Cs_2)), 
}
\end{equation*}
that follows from this direct computation: 
\begin{align*} 
&\Gamma_{\Sigma^\Phi}(\Gg_{\Cs_1}(x))_{[\omega_2]} =
\Lambda^\Phi\Big((\lambda^\Phi)^\bullet(\hat{x})_{[\omega_2]}\Big) 
=\Lambda^\Phi\Big([\omega_2]_{A_2B_2}, \, \hat{x}(\lambda^\Phi([\omega_2]_{A_2B_2}))\Big) \\
&=\Lambda^\Phi\Big([\omega_2]_{A_2B_2},\,  
x + \Is_{\lambda^\Phi([\omega_2]_{A_2B_2})}\Big)  
= \Big([\omega_2]_{A_2B_2}, \, 
\Phi(x) + \Is_{[\omega_2]_{A_2B_2}}\Big) 
=\Gg_{\Cs_2}(\Phi(x))_{[\omega_2]}. 
\end{align*}

To see that the map $\Eg:\Es\mapsto\Eg_\Es$ (that to every spaceoid $(\Es,\pi,\Xs)$ associates its evaluation transform $\Eg_\Es$) 
is a natural isomorphism between the identity endofunctor $\I_\Tf:\Tf\to\Tf$ and 
the functor $\Sigma\circ\Gamma:\Tf\to\Tf$ we must prove, for any given morphism of spaceoids $(f,\F)$ from $(\Es_1,\pi_1,\Xs_1)$ 
to $(\Es_2,\pi_2,\Xs_2)$, the commutativity of the diagram: 
\begin{equation*}
\xymatrix{
(\Es_1,\pi_1,\Xs_1) \ar[rrr]^{\Eg_{\Es_1}=(\eta^{\Es_1},\Omega^{\Es_1})} \ar[d]_{(f,\F)} & & & \Sigma(\Gamma(\Es_1))
\ar[d]^{\Sigma^{\Gamma_{(f,\F)}}=(\lambda^{\Gamma_{(f,\F)}},\Lambda^{\Gamma_{(f,\F)}})} \\
(\Es_2,\pi_2,\Xs_2) \ar[rrr]^{\Eg_{\Es_2}=(\eta^{\Es_2},\Omega^{\Es_2})} & & & \Sigma(\Gamma(\Es_2)).
} 
\end{equation*} 
The proof amounts to showing the equalities
\begin{gather}
\lambda^{\Gamma_{(f,\F)}}\circ \eta^{\Es_1}=
\eta^{\Es_2}\circ f, 
\quad  
\Omega^{\Es_1}\circ (\eta^{\Es_1})^\bullet(\Lambda^{\Gamma_{(f,\F)}}) \circ \Theta_1
= \F\circ f^\bullet(\Omega^{\Es_2})\circ \Theta_2, \label{eq: e2}
\end{gather}
where  $\Theta_1:= \Theta^{\Es^{\Gamma(\Es_2)}}_{\lambda^{\Gamma_{(f,\F)}}, \eta^{\Es_1}}$, 
$\Theta_2:=\Theta^{\Es^{\Gamma(\Es_2)}}_{\eta^{\Es_2}, f}$. 

Since for every point $p_{AB}\in\Xs_1$, we have 
$\lambda^{\Gamma_{(f,\F)}}\circ \eta^{\Es_1}(p_{AB})=( [\gamma_p\circ\ev_p\circ\Gamma_{(f,\F)}], f_\Rs(AB))$ and 
$\eta^{\Es_2}\circ f(p_{AB})=([\gamma_{f(p)}\circ\ev_{f(p)}], f_\Rs(AB))$, the first equation is a consequence of 
proposition~\ref{prop: field}. The second equation is then proved by a lengthy but elementary calculation. 
\end{proof}

\medskip 

\begin{remark}
Finally, note that, although for simplicity we only described a spectral theory for commutative full C*-categories, it is perfectly viable and there are no substantial obstacles to the development of a spectral theory for commutative full ``non-unital''  
C*-categories\footnote{Strictly speaking these are not categories, since they are lacking identities, but they otherwise satisfy all the other properties listed in the definition of a C*-category.} (as defined by P.~Mitchener~\cite{M1}). In this case the base spectrum is only locally compact and we have to deal with a locally compact version of topological spaceoids (so, for example, only sections ``going to zero at infinity'' are considered in the definition of the section functor). 
\end{remark} 

\subsection{Horizontal Categorification}\label{sec: hc+}

The usual Gel'fand-Na\u\i mark duality theorem is easily recovered from our result identifying a compact Hausdorff topological space $X$ with the trivial spaceoid $\Ts_X$ with total space $\Xs_X\times \CC$ and base category 
$\Xs_X:=\Delta_X\times \Rs_{\O_X}$ where $\O_X:=\{X\}$ is a discrete space with only one point $X$; and similarly, identifying a unital commutative C*-algebra $\As$ with the full commutative C*-category $\Cs_\As$ with one object via 
$\Hom_{\Cs_\As}:=\As$ and $\Ob_{\Cs_\As}:=\{\As\}$. 
More precisely, the duality $(\Gamma,\Sigma)$ between the categories $\Tf$ and $\Af$ is a ``horizontal categorification'' of the usual Gel'fand-Na\u\i mark duality in the sense specified by the following result whose proof is absolutely elementary: 
\begin{theorem} 
Let $\Tf^{(1)}$ denote the full subcategory of $\Tf$ consisting of those trivial spaceoids $\Ts_X:=\Xs_X\times \CC$, where 
$\Xs_X:=\Delta_X\times \Rs_{\Os_X}$, $\Os_X:=\{X\}$ and $X$ is a compact Hausdorff space. 
Let $\Af^{(1)}$ denote the full subcategory of $\Af$ consisting of those full commutative small one-object C*-categories $\Cs_\As$ with morphisms $\Hom_{\Cs_\As}:=\As$, objects $\Ob_{\Cs_\As}:=\{\As\}$ and composition involution and norm induced from those in the commutative unital C*-algebra $\As$. 
The natural duality $(\Gamma,\Sigma)$ between the categories $\Tf, \Af$ restricts to a duality $(\Gamma^{(1)},\Sigma^{(1)})$ between the categories $\Tf^{(1)}, \Af ^{(1)}$ i.e.~the
following pair of diagrams of functors is commutative: 
\begin{equation*}
\xymatrix{
\Tf^{(1)} \ar@{^(->}[d] \ar@^{->}[rr]^{\Gamma^{(1)}} & & \Af^{(1)} \ar@^{->}[ll]^{\Sigma^{(1)}} \ar@{^(->}[d] \\
\Tf \ar@^{->}[rr]^\Gamma & & \ar@^{->}[ll]^\Sigma \Af, 
}
\end{equation*}
where $\Gamma^{(1)}$ and $\Sigma^{(1)}$ denote the restrictions of the functors $\Gamma$ and $\Sigma$ and the vertical arrows denote the inclusion functors of the respective categories. 

The category $\Tf_1$ of continuous maps between compact Hausdorff spaces is isomorphic to the category $\Tf^{(1)}$ via the functor $\Ts: \Tf_1\to\Tf^{(1)}$ defined as follows:
\begin{itemize}
\item
to every compact Hausdorff space $X$, $\Ts$ associates the spaceoid $\Ts_X:=(\Es_X,\pi_X,\Xs_X)$ that is the trivial bundle with fiber $\CC$ over the space $\Xs_X:=\Delta_X\times\{(X,X)\}$, 
\item
to every continuous map $f: X\to Y$ between compact Hausdorff spaces, $\Ts$ associates the morphism of spaceoids $\Ts(f):\Ts_X\to \Ts_Y$ defined by $\Ts(f):=(\Ts(f)_\Xs,\Ts(f)_\Es)$ where $\Ts(f)_\Xs:p_{XX}\mapsto f(p)_{YY}$, for all 
$p\in X$, and $\Ts(f)_\Es:\Ts(f)_\Xs^\bullet(\Es_Y)\to\Ts_X$ denotes the canonical isomorphism between trivial line bundles over 
$X$.\footnote{Remember that the pull-back of the trivial line bundle $\Ts_Y$ under the homeomorphism $\Ts(f)_\Xs$ is a trivial line bundle on $\Xs_X$.}
\end{itemize}
The category $\Af_1$ of unital $*$-homomorphisms of unital commutative C*-algebras is isomorphic to the category $\Af^{(1)}$ via the functor $\Cs:\Af_1\to\Af^{(1)}$ that to every unital commutative \hbox{C*-algebra} $\As$ associates the C*-category $\Cs_\As$ and that to every unital $*$-homomorphism $\phi:\As\to\Bs$ associates the $*$-functor $\Cs(\phi):\Cs_\As\to\Cs_\Bs$ given on arrows by $\Cs(\phi)(x):=\phi(x)$, for all $x\in \As$, and on objects by $\Cs(\phi)_o:\As\mapsto \Bs$. 

The functors $\Cs\circ\Gamma$ and $\Gamma^{(1)}\circ\Ts$ are naturally equivalent via the natural transformation that to every 
$X$ associates the canonical isomorphism between $\Cs_{C(X;\CC)}$ and $\Gamma(\Ts_X)$. 

The functors $\Ts\circ \Sigma$ and $\Sigma^{(1)}\circ \Cs$ are naturally equivalent via the natural transformation that to every $\As$ associates the canonical isomorphism between $\Ts_{\Sp(\As)}$ and $\Sigma(\Cs_\As)$. 
\end{theorem} 

\section{Examples and Applications} \label{sec: exappl}

Commutative full C*-categories are abundant, just to mention a few examples: 
\begin{itemize} 
\item
Every Abelian unital C*-algebra $\As$ gives a commutative full C*-category $\Cs_\As$ with only one object (as already mentioned in subsection~\ref{sec: hc+}). 
\item 
Examples of commutative full C*-categories with two objects can be obtained via the following construction 
(see~L.~Brown-P.~Green-M.~Rieffel~\cite{BGR}) of the ``linking \hbox{C*-category}'' $\Ls(\Ms)$ of an imprimitivity Hilbert 
C*-bimodule $_\As\Ms_\Bs$ over two commutative unital C*-algebras $\As,\Bs$.    
Let $_\As\Ms_\Bs$ be an imprimitivity C*-bimodule over unital commutative C*-algebras $\As,\Bs$. Denote by $_\Bs\Ms^+_\As$ its  Rieffel dual and by $\iota:\Ms\to\Ms^+$ the canonical bijective map such that $\iota(a\cdot x\cdot b)=b^*\cdot\iota(x)\cdot a^*$ for all $a\in \As, b\in \Bs, x\in \Ms$ (see for example~\cite[Proposition~2.19]{BCL3} for more details). 
The linking C*-category $\Ls(\Ms)$ of $\Ms$ is the full commutative C*-category with two objects $A,B$ and morphisms given by 
$\Ls(\Ms)_{AA}:=\As, \ \Ls(\Ms)_{BB}:=\Bs, \ \Ls(\Ms)_{AB}:=\Ms, \ \Ls(\Ms)_{BA}:=\Ms^+$ where the only non-elementary operations are the involutions of elements of $\Ms$ and $\Ms^+$, given by 
$x^*:=\iota(x)$ and $y^*:=\iota^{-1}(y)$ for all $x\in \Ms$, $y\in \Ms^+$ 
and the compositions between elements of $\Ms$ and $\Ms^+$ that are given via their respective $\As$-valued and $\Bs$-valued inner products as follows: $x\circ y:={}_\As\ip{x}{y^*}$ and $y\circ x:=\ip{y^*}{x}_\Bs$, for all $x\in \Ms$, $y\in \Ms^+$. 

Making use of the canonical isomorphisms of imprimitivity bimodules 
\begin{equation*}
(\Ms\otimes\Ns)\otimes\Ts\simeq(\Ms\otimes\Ns\otimes\Ts)\simeq\Ms\otimes(\Ns\otimes\Ts)\quad  \text{and}\quad  
(\Ms\otimes\Ns)^+\simeq\Ns^+\otimes\Ms^+
\end{equation*}  
for the definition of the compositions via tensor products and ``contractions'', we can generalize the previous construction of linking C*-category to the case of an arbitrary (finite) collection of objects. 
In practice, given a (finite) family of commutative unital C*-algebras $\As_1,\dots,\As_n$ and a family of imprimitivity Hilbert C*-bimodules
$_{\As_1}\Ms_{\As_2},\dots,\ _{\As_{n-1}}\Ms_{\As_n}$, the ``linking C*-category'' $\Ls(_{\As_1}\Ms_{\As_2},\dots,\ _{\As_{n-1}}\Ms_{\As_n})$ is the full commutative C*-category with $n$ objects $B_1,\dots,B_n$, where 
\begin{gather*}
\Ls(_{\As_1}\Ms_{\As_2},\dots,\ _{\As_{n-1}}\Ms_{\As_n})_{B_jB_k}:=
\begin{cases}
{}_{\As_j}\Ms_{\As_{j+1}}\otimes\cdots\otimes{}_{\As_{k-1}}\Ms_{\As_k}, \quad \text{for $j<k$}, 
\\
\As_j, \quad \text{for $j=k$}, 
\\
({}_{\As_k}\Ms_{\As_{k+1}}\otimes\cdots\otimes{}_{\As_{j-1}}\Ms_{\As_j})^+, \quad \text{for $k<j$}, 
\end{cases}
\end{gather*} 

The examples in the next item are rather natural, especially for those who are familiar with the Doplicher-Roberts abstract duality theory for compact groups.

\item
Let $G$ be a compact group, and consider the C*-category $\Rep (G)$ with objects the unitary representations of $G$ on Hilbert spaces and arrows their intertwiners (actually, $\Rep (G)$ is a W*-category). 
Then the full subcategory of $\Rep (G)$ whose objects are the multiplicity-free representations is commutative.
Moreover, a category of multiplicity-free representations is full whenever all the objects are equivalent.

Let $\As$ be a unital C*-algebra. A category of nondegenerate $*$-representations of $\As$ on Hilbert spaces
is commutative whenever all the objects are multiplicity-free (see e.g.~\cite[Chapter~ 2]{Ar}). In addition, such a commutative 
W*-category is full whenever all the objects are equivalent. 
If $\As$ is commutative with metrizable spectrum, 
a category of nondegenerate representations of $\As$ on separable Hilbert spaces that is both commutative and full
can be interpreted in terms of a family of equivalent finite Borel measures on the spectrum of 
$\As$~\cite[Theorems 2.2.2 and 2.2.4]{Ar}.
This fact can be generalized to GCR algebras~\cite[Chapter~4]{Ar}.  

Let $\As$ be a unital C*-algebra, and consider the C*-category $\End (\As)$ 
with objects the unital $*$-endomorphisms of $\As$
and Banach spaces of arrows 
\begin{equation*}
(\rho,\sigma) = \{x \in \As \ | \ x\rho(a)=\sigma(a)x, \ \forall a \in \As\}.
\end{equation*}
The category $\Aut (\As)$, i.e.~the full subcategory of $\End(\As)$ with objects the unital 
\hbox{$*$-automorphisms} of $\As$, is clearly commutative, as $(\rho,\rho)$ equals the center of $\As$, 
for every automorphism $\rho$.\footnote{Of course, this is still true for all irreducible morphisms.}
Again, a subcategory of $\Aut (\As)$ is full whenever all the objects are equivalent.
\item 
P.~Mitchener~\cite{M1} associates C*-categories $C^*(\G)$ and $C^*_r(\G)$ to a discrete groupoid $\G$. 
It is easy to see that these categories are commutative exactly when all the stabilizer subgroups of the groupoid $\G$ are Abelian (i.e.~$\G_{AA}$ is an Abelian group for all objects $A$ of the groupoid). In that case they are full if the groupoid $\G$ is transitive (i.e.~$\G_{AB}\neq\varnothing$, for every pair of objects $A,B\in \Ob_\G$). 
 
This example admits an immediate generalization to the case of involutive categories. Given an involutive category $\X$, the set of 
$\CC$-valued maps on $\X$ with finite support contained in any one of the sets $\X_{AB}$, with $A,B\in \Ob_\X$, is the family of morphisms of a $*$-category $C^*_o(\X)$, with objects $\Ob_\X$, where the composition is the usual convolution of finite sequences and the involution is defined via $(\alpha_x)^*:=\cj{\alpha}_{x^*}$. 
The $*$-category $C^*_o(\X)$ has a natural continuous left-regular action on $L^2(\X)$ (that is the family of Hilbert spaces, indexed by $A\in \Ob_\X$, obtained by completing $\oplus_{B\in \Ob_\X}C^*_o(\X)_{AB}$ under the inner product 
$\ip{(\alpha_x)}{(\beta_y)}:=\sum_{x,y}\cj{\alpha}_{x^*}\cdot\beta_y$) and its C*-completion in the induced operator norm is the C*-category $C^*_r(\X)$. 
Taking the C*-completion of $C^*_o(\X)$ under the supremum of all the C*-norms induced by its continuous representations we obtain the C*-category $C^*(\X)$. 
The categories $C^*(\X)$ and $C^*_r(\X)$ are commutative whenever $\X_{AA}$ is commutative for all objects $A\in \Ob_\X$ and they are full if and only if the category $\X$ is saturated in the following sense $\X_{AB}\circ\X_{BC}=\X_{AC}$ for all objects 
$A,B,C\in \Ob_\X$. 

\item 
Given any non-diagonal arrow $x$ in a C*-category $\Cs$, the C*-subcategory $\Cs(x)$ of $\Cs$ generated by $x$
is full and commutative, see theorem~\ref{th: hcatcal} and the related discussion 
(notice that $\Cs(x)$ might well be a non-unital C*-category if $x$ is not invertible). 
\item The category of Hermitian line bundles over a compact Hausdorff space $X$ with line bundle morphisms as arrows is a 
C*-category, which turns out to be full and commutative. 
\end{itemize}

We now deal with specific examples and constructions of spaceoids. Note that (although every topological spaceoid is of course isomorphic to the spectrum of a commutative full C*-category) the examples mentioned here below have in principle no direct relation with C*-categories and arise from some well-known constructs in (algebraic) topology. 

\begin{itemize}
\item 
As already described in detail in subsection~\ref{sec: hc+}, the most elementary examples of spaceoid are those associated to every compact Hausdorff topological spaces $X$ via the trivial Hermitian line bundles $\Ts_X:=(\Es_X,\pi_X,\Xs_X)$ over the topological space $\Xs_X:=\Delta_X\times \{(X,X)\}$ with total space $\Es_X:=\Xs_X\times \CC$ and projection $\pi_X$ onto the first factor. 
\item 
Every (possibly non-trivial) Hermitian line bundle $(E,\pi,X)$ over a compact Hausdorff topological space $X$ uniquely determines a spaceoid $\Ls(E):=(\Es^E,\pi^E,\Xs^E)$, called its ``linking spaceoid'', in the following canonical way. 
Define the base topological involutive category as $\Xs^E:=\Delta_X\times\Rs_\O$ with $\O:=\{A,B\}$, and as total space consider 
$\Es^E_{AA}:=(\Delta_X\times \{AA\})\times \CC$, $\Es^E_{BB}:=(\Delta_X\times \{BB\})\times \CC$, 
$\Es^E_{AB}:=E\times\{AB\}$, $\Es^E_{BA}:=E^+\times\{BA\}$, where by $(E^+,\pi^+,X)$ we denote the Hermitian line bundle dual to $(E,\pi,X)$ (this is the line bundle with fibers $E^+_p:=(E_p)^+$ given by the dual of the inner product vector space $E_p$). 
Define on the total space $\Es^E$ the operations of involution as usual fiberwise conjugation on $\Es^E_{AA}$ and $\Es^E_{BB}$ and by the canonical antilinear map induced between $\Es^E_{AB}$ and $\Es^E_{BA}$ by the natural fiberwise anti-isomorphism between $E$ and $E^+$. Finally define the composition on the total space as the usual fiberwise product on $\Es^E_{AA},\Es^E_{BB}$ and, between elements in $\Es^E_{AB}$ and $\Es^E_{BA}$, via the canonical contraction between $E$ and $E^+$ as 
$e_{AB} \circ e'_{BA}:=e'(e)_{AA}$ and $e'_{BA}\circ e_{AB}:=e'(e)_{BB}$, for all $e\in E_p$ and $e'\in E^+_p$, with $p\in X$. 
\item 
In perfect parallel with the construction of the linking C*-category for a family of Hilbert C*-bimodules, the previous construction of the linking spaceoid of a Hermitian line bundle can be generalized in order to define the ``linking spaceoid'' $\Ls(E^1,\dots,E^n)$ of a family of (possibly non-trivial) Hermitian line bundles $(E^1,\pi^1,X), \dots, (E^n,\pi^n,X)$ over the same compact Hausdorff space 
$X$. 

For this purpose, denoting $\Ls(E^1,\dots, E^n):=(\Es^{E^1\dots E^n},\pi^{E^1\dots E^n},\Xs^{E^1\dots E^n})$, we take the base topological $*$-category as $\Xs^{E^1\dots E^n}:=\Delta_X\times \Rs_\O$, where $\O:=\{A_1,\dots,A_{n+1}\}$ is a set of 
$n+1$ elements. We then define the total space $\Es^{E^1\dots E^n}$ of the linking spaceoid specifying its blocks on the topological space $\Xs^{E^1\dots E^n}_{A_jA_k}=\Delta_X\times\{A_jA_k\}$ as follows: 
\begin{equation*}
\Es^{E^1\dots E^n}_{A_jA_k}:=
\begin{cases}
E^j\otimes \cdots \otimes E^{k-1}, \quad \text{for $j<k-1$},
\\
E^j \quad \text{for $j=k-1$}, 
\\
\CC\times (\Delta_X\times \{A_jA_k\}), \quad \text{for $j=k$}, 
\\
(E^j)^+ \quad \text{for $k=j-1$}, 
\\
(E^{k-1})^+\otimes \cdots\otimes (E^j)^+ \quad \text{for $k<j-1$},
\end{cases}
\end{equation*}
where $\otimes$ denotes the fiberwise tensor products of line bundles over the same space $X$. 
On the total space $\Es^{E^1\dots E^n}$ the fiberwise involution and the composition are defined making use of the canonical isomorphisms of line bundles $(E^j\otimes E^k)^+\simeq (E^k)^+\otimes (E^j)^+$ and $E^i\otimes (E^j\otimes E^k)\simeq E^i\otimes E^j\otimes E^k\simeq (E^i\otimes E^j)\otimes E^k$, for all $i,j,k=1,\dots n$, via the contraction dualities between $E^j$ and $(E^j)^+$ for $j=1,\dots, n$ and the tensor products $E^j\times E^k\to E^j\otimes E^k$ for all $j,k=1,\dots,n$. 

Examples of linking spaceoids of Hermitian line bundles, that stay in perfect duality with those of the linking C*-categories of imprimitivity Hilbert C*-bimodules previously described, can be obtained via a ``bivariant'' notion of Hermitian line bundle (i.e.~a Hermitian line bundle on a base space that is the graph $f\subset X\times Y$ of a  homeomorphism $f:X\to Y$ between two compact Hausdorff topological spaces $X,Y$) that is developed in more detail in our companion work~\cite{BCL3}. 
\item 
Finally, we briefly introduce here another natural way to produce spaceoids via ``associated line bundles'' to a suitable categorification of $\TT$-torsors. In more details: given an equivalence relation $\Rs$, consider the family $[\Rs;\TT]$ of homomorphisms of the groupoid $\Rs$ with values in the torus group $\TT:=\{\alpha\in \CC \ | \ |\alpha|=1\}$. Clearly $[\Rs;\TT]$ is itself an Abelian group with the operation of pointwise multiplication between homomorphisms. 
For any compact Hausdorff space $X$ consider a $[\Rs;\TT]$-torsor $(\Ts,\pi,X)$. Since the set $[\Rs;\CC]$ of $\CC$-valued homomorphisms has a natural structure of $[\Rs;\TT]$-space with action given by pointwise multiplication, we can construct the ``associated bundle'' $\Ts\times_{[\Rs;\TT]}[\Rs;\CC]$ over the space $X$ whose elements are equivalence classes $[(\phi,v)]$ of pairs $(\phi,v)\in \Ts\times[\Rs;\CC]$ under the equivalence relation $(\phi,v)\simeq(\psi,w)$ if and only if there exists $g\in [\Rs;\TT]$ such that $\phi\cdot g=\psi$ and $v=g\cdot w$. Every such ``associated bundle'' can be seen, simply by rearrangement of variables,   as a spaceoid over the base $\Delta_X\times\Rs$. To obtain the spectral spaceoid of a full commutative C*-category $\Cs$, it is sufficient to take $\Ts:=[\Cs;\CC]$, the set of $\CC$-valued $*$-functors on $\Cs$.
\end{itemize} 

Spectral spaceoids can be easily ``assembled'' starting from the spectra of imprimitivity C*-bimodules developed (via Serre-Swan theorem) in~\cite[Theorem~3.1]{BCL3}. In every full commutative C*-category $\Cs$ and for every pair of objects $A,B\in \Ob_\Cs$, the spectrum of the imprimitivity C*-bimodule $\Cs_{AB}$ is a Hermitian line bundle $(E_{BA},\pi_{BA},R_{BA})$ on the graph\footnote{Note that here we are using $R_{BA}$ to denote the homeomorphism $R_{BA}:\Sp(\Cs_{AA})\to\Sp(\Cs_{BB})$ and together its graph $R_{BA}\subset \Sp(\Cs_{AA})\times\Sp(\Cs_{BB})$. More generally, we use the same letter $R$ to denote a relation from the set $A$ to the set $B$ and its graph $R\subset A\times B$.} of a unique homeomorphism $R_{BA}:\Sp(\Cs_{AA})\to\Sp(\Cs_{BB})$ between the Gel'fand spectra of the two unital commutative C*-algebras  $\Cs_{AA}$ and $\Cs_{BB}$. 
\\ 
Now the (necessarily disjoint) union of all the graphs $R_{BA}\subset \Sp(\Cs_{AA})\times\Sp(\Cs_{BB})$ of the homeomorphisms 
$R_{BA}$, with $A,B\in \Ob_\Cs$, can be seen as the graph of a new relation $\left(\bigcup_{A, B\in \Ob_\Cs}R_{BA}\right)\subset \left(\bigcup_{B\in \Ob_\Cs}\Sp(\Cs_{BB})\right)\times\left(\bigcup_{A\in \Ob_\Cs}\Sp(\Cs_{AA})\right)$ in the set $\bigcup_{A\in \Ob_\Cs}\Sp(\Cs_{AA})$ that is the ``disjoint union'' of the Gel'fand spectra of the diagonal C*-algebras $\Cs_{AA}$ with $A\in \Ob_\Cs$.
Since the homeomorphisms $R_{BA}$ are given by $R_{BA}=|_{BB}\circ |_{AA}^{-1}$ in terms of the restriction homeomorphisms $|_{AA}:\Sp_b(\Cs)\to\Sp(\Cs_{AA})$ defined in propositions~\ref{prop: bije} and theorem~\ref{prop: homo-sp}, the relation 
$\bigcup_{A, B\in \Ob_\Cs}R_{BA}$ is an equivalence relation.\footnote{Note that the new equivalence relation 
$\bigcup_{A,B\in \Ob_\Cs}R_{BA}$ is a relation between elements of the union $\bigcup_{A\in \Ob_\Cs} \Sp(\Cs_{AA})$ of the Gel'fand spectra not to be confused with the ``coarse grained'' equivalence relation 
$\{(\Sp(\Cs_{AA}),\Sp(\Cs_{BB}))\ | \ A,B\in \Ob_\Cs\}$ between the Gel'fand spectra themselves or equivalently with the 
sub-equivalence relation $\{R_{AB}\ | \ A,B\in\Ob_\Cs\}$ of the groupoid of homeomorphisms of compact Hausdorff spaces.} Furthermore, the ``disjoint union'' of all such Hermitian line bundles 
$(E_{BA},\pi_{BA},R_{BA})$ becomes a new Hermitian line bundle $\biguplus_{A,B\in \Ob_\Cs}(E_{AB},\pi_{AB},R_{AB}):=\left(\bigcup_{A, B\in \Ob_\Cs}E_{BA},\bigcup_{A, B\in \Ob_\Cs}\pi_{BA},\bigcup_{A, B\in \Ob_\Cs}R_{BA}\right)$ with total space 
$\bigcup_{B,A\in \Ob_\Cs}E_{BA}$ and base space $\bigcup_{A, B\in \Ob_\Cs}R_{BA}$. 

The map $\tau: [\omega]_{AB}\mapsto (\omega_{AA},\omega_{BB})$ provides an isomorphism of topological involutive categories between $\Delta^\Cs\times\Rs^\Cs$ and $\bigcup_{AB\in \Ob_\Cs}R_{BA}$. 
Since, as Hilbert spaces, the fibers $E_{(\omega_{AA},\omega_{BB})}$ are given by 
$E_{(\omega_{AA},\omega_{BB})}:=\Cs_{AB}/(\Cs_{AB}\cdot\ke(\omega))$ and so coincide with the fibers of the spectral spaceoid 
$\Sigma(\Cs)$ on the elements $[\omega]_{AB}$, we see that $\bigcup_{A,B\in \Ob_\Cs}E_{(\omega_{AA},\omega_{BB})}$ can be naturally equipped with the structure of C*-category and so the bundle $\biguplus_{A,B\in \Ob_\Cs}(E_{AB},\pi_{AB},R_{AB})$ is a rank-one Fell bundle. 
Finally we see that the spectral spaceoid $\Sigma(\Cs)$ coincides with the $\tau$-pull-back 
$\tau^\bullet(\bigcup_{A,B\in \Ob_\Cs}E_{AB})$ of the rank-one Fell bundle $\biguplus_{A,B\in \Ob_\Cs}(E_{AB},\pi_{AB},R_{AB})$. 

\bigskip 

The classical Gel'fand-Na\u\i mark duality for commutative C*-algebras had a number of important applications and, in a parallel way, its extension for commutative full C*-categories described here, will provide interesting ``horizontal categorifications'' of those applications. Among them there are, for example:
\begin{itemize} 
\item 
a Fourier transform in the context of Pontryagin duality for commutative groupoids, 
\item 
a continuous functional calculus for bounded linear operators between Hilbert spaces, 
\item 
a spectral theorem for bounded linear operators between Hilbert spaces. 
\end{itemize}  
Most of these ideas will actually require a serious amount of additional work that deserves a separate detailed treatment elsewhere and so, in order to exemplify the ``capabilities'' of our result, we limit ourselves to the development of a ``horizontally categorified continuous functional calculus'' which is the most immediate and straightforward of the previous topics.   

\medskip

Let $\Cs$ be a C*-category, not necessarily commutative or full, and let $x\in \Cs_{AB}$ be one of its morphisms. 
Consider now $\Cs(x)$, the (non-necessarily unital) C*-category generated by $x$. 
By definition, this is the C*-subalgebra of $\Cs_{AA}$ generated by $x$ if $A=B$, 
and the \hbox{C*-subcategory} of $\Cs$ with two objects $A$ and $B$ and arrow spaces
\begin{gather*}
\Cs(x)_{AA}=\spa\{(x\circ x^*)^n \ | \ n=1,2,3,\ldots\}^-, 
\\
\Cs(x)_{BB}=\spa\{(x^*\circ x)^n \ | \ n=1,2,3,\ldots\}^- 
\end{gather*}
and $\Cs(x)_{AB}=x\circ\Cs(x)_{BB}=\Cs(x)_{BA}^*$ otherwise.
Notice that $\Cs(x) (= \Cs(x^*))$ is always full. 
If $A\neq B$ the C*-category $\Cs(x)$ is always commutative and for $A=B$ it is commutative if and only if $x$ is normal and in these two cases we can immediately apply our spectral results on the horizontal categorified Gel'fand transform to realize that: every morphism in the category $\Cs(x)$ is uniquely described by a continuous section of a ``block'' of the spectral spaceoid of $\Cs(x)$. In more detail we have:
\begin{theorem}[Horizontally categorified continuous functional calculus] \label{th: hcatcal}
Let $x\in \Cs_{AB}$ be an element of a C*-category $\Cs$ and let $\Cs(x)$ denote the (non-necessarily unital) C*-category generated by $x$ inside $\Cs$. If either the objects $A$ and $B$ are different, 
or $A=B$ and the element $x\in \Cs_{AA}$ is normal,
then the C*-category $\Cs(x)$ is full and commutative.  

In that case, for every continuous section $\sigma\in\Gamma(\Sigma(\Cs(x)))_{AB}$ of the block $AB$ of the spectral spaceoid of 
$\Cs(x)$, there is an associated element $\sigma(x)\in \Cs(x)_{AB}$. 

Moreover, the resulting map $\Fg_x:\sigma\mapsto \sigma(x)$ is an isometric $*$-functor from the (possibly non-unital) 
C*-category $\Gamma(\Sigma(\Cs(x)))$ onto $\Cs(x)\subset \Cs$.   
\end{theorem}
\begin{proof} 
Given $x\in \Cs_{AB}$, we simply define 
$\Fg_x:\Gamma(\Sigma(\Cs(x)))\to \Cs$ as the map given by the inverse of the Gel'fand transform 
$\Gg:\Cs(x)\to\Gamma(\Sigma(\Cs(x)))$ i.e.~$\Fg_x:=\Gg_{\Cs(x)}^{-1}$.  
\end{proof} 
We call the $*$-functor $\Fg_x$ the \emph{continuous functional calculus} of $x$.

Note that the previous result might open the way to obtaining a spectral representation (and hence a spectral theorem) also for bounded linear operators that are not normal. In fact if $T:H\to H$ is an arbitrary bounded linear operator on a Hilbert space $H$, we can always regard $T$ as a morphism in an off-diagonal block of the C*-category, with two objects, of bounded linear operators between $H_A:=H=:H_B$.  

\section{Outlook} \label{sec: outlook}

We have introduced commutative C*-categories and started a program for their ``topological description'' in terms of their spectra, here called spaceoids.

In particular, we have obtained a Gel'fand-type theorem for full commutative C*-categories. 
Although the statement of the main result (theorem~\ref{th: catgel}) looks extremely natural, our proofs mostly rely on a ``brute force'' exploitation of the underlying structure and more streamlined arguments are likely to be found.
Also, the result by itself is not as general as possible and certainly it leaves room for extensions in several directions, still hopefully we have provided some insight about how to achieve them.  

For instance, we have only considered the case of $*$-functors between (full, commutative) C*-categories that are bijective on the objects.
(Of course, this trivially includes morphisms between commutative \cs-algebras).
In the next step, one would like to treat the case of $*$-functors that are not bijective on the objects. 
We tend to believe that this should not require significant modifications of our treatment and possibly could be dealt with using relators (that we introduced in~\cite{BCL}).\footnote{For this purpose it should be enough to introduce a category of spaceoids in which the morphisms $f:\Xs_1\to\Xs_2$ between the two base $*$-categories $\Xs_1=\Delta_1\times\Rs_1$ and 
$\Xs_2=\Delta_2\times\Rs_2$ are given by 
$*$-relators $f:=(f_\Delta,f_\Rs)$ where now the $*$-functor $f_\Rs:\Rs_2\to\Rs_1$ is acting in the ``reverse direction''.}

Perhaps a more important point would be to remove the condition of fullness.
At present we have not discussed the issue in detail, but certainly the information that we have already acquired should significantly simplify the task.

Also, along the way, we have somehow taken advantage of our prior knowledge of the Gel'fand and Serre-Swan theorems. Eventually one would like to provide more intrinsic proofs directly in the framework of C*-categories (possibly unifying and extending both Gel'fand and Serre-Swan theorems in a ``strict $*$-monoidal'' version of Takahashi theorem~\cite{Ta1,Ta2}). 
In this respect, it looks promising to work directly with module categories.
Besides, it is somehow disappointing that to date, for $X$ and $Y$ compact Hausdorff spaces, there seems to be no available general classification result for $C(X)$-$C(Y)$-bimodules. 

\medskip

Hilbert C*-bimodules that are not-necessarily imprimitivity bimodules 
should definitely play a role when discussing a classification result for generally non-commutative C*-categories, possibly along the lines of a generalization to C*-categories of the Dauns-Hofmann theorem for C*-algebras~\cite{DH}. 
One might also explore possible connections with the non-commutative Gel'fand spectral theorem of 
R.~Cirelli-A.~Mani\`a-L.~Pizzocchero~\cite{CMP} and the subsequent non-commutative Serre-Swan duality by 
K.~Kawamura~\cite{Kaw} and E.~Elliott-K.~Kawamura~\cite{EK}. 
Similarly, it might be very interesting to investigate the connections between our spectral spaceoids and other spectral notions such as locales and topoi already used in the spectral theorems by 
B.~Banachewki-C.~Mulvey~\cite{BM} and C.~Heunen-K.~Landsmann-B.~Spitters~\cite{HLS}.  

\medskip

In the long run, one would like to (define and) classify commutative Fell bundles over suitable involutive categories. The notion of a Fell bundle could be even generalized to that of a fibered category enriched over another ($*$-monoidal) category.

\medskip

Needless to say, one should analyze more closely the mathematical structure of spaceoids, introduce suitable topological invariants, study their symmetries, \dots, and investigate relations to other concepts that are widely used in other branches of mathematics, e.g.~in algebraic topology/geometry as well as in gauge theories.
Some geometric structures could become apparent when considering the representation of spaceoids as continuous fields of (one-dimensional commutative) 
C*-categories as discussed by E.~Vasselli in~\cite{V4}. 

\medskip

The Gel'fand transform for general commutative \cs-categories raises several questions  (undoubtedly it could be defined for more general Banach categories, leading to a wide range of possibilities for further studies).

In particular, an immediate application would yield a Fourier transform and accordingly a reasonable concrete duality theory for commutative discrete groupoids 
(see~M.~Amini~\cite{Am} for another approach that applies to compact but-not-necessarily-commutative-groupoids,
T.~Timmermann~\cite{Ti} for a more abstract setup and G.~Goehle~\cite{Go} for a discussion of duality for
locally compact Abelian group bundles).  

\medskip

As far as we are concerned, our main motivation to work with C*-categories comes from analyzing the categorical structure of non-commutative geometry (where morphisms of ``non-commutative spaces'' are given by bimodules) and one is naturally led to speculating about the possible evolution of the notion of spectra and morphisms in A.~Connes'  non-commutative geometry 
(cf.~\cite{BCL, BCL0, CCM}). 
In this direction, some of the first questions that come to mind are:
\begin{itemize}
\item[]
Is there a suitable notion of spectral triple over a C*-category?
\item[] 
Is it possible to consider a horizontal categorification of a spectral triple?
\end{itemize}

Of course this represents only the starting point for a much more ambitious program aiming at a ``vertical categorification'' of the notion of spectral triple\footnote{The need for a notion of ``higher spectral triple'' has been already advocated by 
U.~Schreiber~\cite{Sch}.} and from several fronts (see for example~\cite{DTT} and also the very detailed discussion by 
J.~Baez~\cite{B2} on the weblog ``The $n$-Category Caf\'e'') 
it is mounting the evidence that a suitable notion of non-commutative calculus necessarily require a higher (actually $\infty$) categorical setting. 
 
In this respect, it seems reasonable to look for a Gel'fand theorem that applies to (strict) commutative higher categories 
(cf.~\cite{Ko}). 
A suitable definition of strict $n$-C*-categories (cf.~\cite{Z} for the case $n=2$) and the proof of a categorical Gel'fand duality (at least for ``commutative'' full strict $n$-C*-categories) are topics that have recently attracted our attention~\cite{BCLS}. 

\medskip

Finally, in this line of thoughts, one could envisage potential applications of a notion of Gromov-Hausdorff distance (cf.~\cite{Rie2}) for C*-categories. 

\medskip

\emph{Acknowledgments.} 

\smallskip

{\small
We acknowledge the support provided by the Thai Research Fund, grant n.~RSA4780022.  
The main part of this work has been done in the two-year visiting period of R.~Conti to the Department of Mathematics of Chulalongkorn University. 

P.~Bertozzini acknowledges the ``weekly hospitality'' offered by the Department of Mathematics in Chulalongkorn University, where most of the work leading to this publication has been done, from June 2005. 
He also desires to thank A.~Carey and B.~Wang at ANU in Canberra, R.~Street, A.~Davydov and M.~Batanin at Macquarie University in Sydney, K.~Hibbert and J.~Links at the Center for Mathematical Physics in Brisbane, 
A.~Vaz Ferreira at the University of Bologna, F.~Cipriani at the ``Politecnico di Milano'', 
G.~Landi and L.~Dabrowski at SISSA in Trieste, R.~Longo, C.~D'Antoni and L.~Zsido at the University of ``Tor Vergata'' in Rome, W.~Szymanski at the University of Newcastle in Australia, E.~Beggs and T.~Brzezinski at Swansea University, D.~Evans at Cardiff University, A.~D\"oring and B.~Coecke at the CLAP workshop in Imperial College, for the kind hospitality, the much appreciated partial support and most of all for the possibility to offer the talks/seminars where preliminary versions of the results contained in this paper have been announced in October 2006, March/October 2007 and May 2008.  

Finally we thank the two anonymous referees for a very careful reading of the manuscript and for 
suggesting several improvements.
}

\end{document}